\documentclass[a4wide,llpt]{article} 

\usepackage{amssymb,amsmath,graphicx,amsthm,psfrag,pstricks,pslatex,float}
\usepackage[english]{babel}
\usepackage[latin1]{inputenc}
\usepackage[T1]{fontenc}

\newtheorem{theorem}{Theorem}[section]
\newtheorem{lemma}[theorem]{Lemma}
\newtheorem{e-proposition}[theorem]{Proposition}
\newtheorem{corollary}[theorem]{Corollary}
\newtheorem{e-definition}[theorem]{Definition\rm}
\newtheorem{remark}[theorem]{\it Remark\/}

\def\og{\leavevmode\raise.3ex\hbox{$\scriptscriptstyle\langle\!\langle$~}}
\def\fg{\leavevmode\raise.3ex\hbox{~$\!\scriptscriptstyle\,\rangle\!\rangle$}}

\def\vs{\vrule width 0cm height 0.1in depth 0in}
\def\fr#1#2{\frac{\displaystyle\vs #1}{\displaystyle\vs #2}}

\def\bint#1#2{ {\displaystyle\vs \int_{#1}^{#2}} }

\def\ie{\emph{i.e. }}
\def\eg{\emph{e.g. }}
\def\cf{\emph{cf. }}

\def\eps{\epsilon}
\def\modu{\,\mathrm{mod}\,}
\def\nequiv{ \equiv \! \! \!\! \! \! \! \! \backslash \;}
\def\Id{\mathrm{Id}}

\def\jo#1{\mathcal{#1}}

\def\supp#1{\textrm{\raisebox{.5ex}{\mbox{$\underset{#1}{\sup}$}}} \:}

\def\somme#1#2{\overset{#2}{\underset{#1}{\sum}}}
\def\bv{\! \big|}
\def\vide{\varnothing}

\def\sm22#1#2#3#4{\left( \begin{smallmatrix}   #1 & #2 \\   #3 & #4 \\  \end{smallmatrix} \right)}
\def\un{1 \!\! \mathrm{l}} 

\def\dd{\mathrm{d} \!}
\def\del{\partial \!}
\def\db{\bar{\partial} \!}

\def\zb{\overline{z}}

\def\zz{\mathbb{Z}}

\def\rr{\mathbb{R}}
\def\cc{\mathbb{C}}

\def\cp{\mathbb{C}\mathrm{P}}

\def\nr#1{\left\| #1 \right\|}
\def\pnr#1{\| #1 \|}

\def\abs#1{\left\lvert #1 \right\rvert}

\def\nwup#1{\left\| #1 \right\|_{W^{1,p}}}

\def\somme#1#2{\overset{#2}{\underset{#1}{\sum}}}
\def\Somme#1#2{\overset{#2}{\underset{#1}{\displaystyle\vs \sum}}}
\def\inter#1#2{\overset{#2}{\underset{#1}{\cap}}}

\def\ssi{\Leftrightarrow}
\def\imp{\Rightarrow}
\def\inj{\hookrightarrow}
\def\surj{\twoheadrightarrow}

\def\img{\mathrm{Im}\,}
\def\diam{\mathrm{Diam}\,}
\def\dim{\mathrm{dim}\,}

\def\tg{\mathrm{T}}
\def\exp{\mathrm{exp}}

\def\wup{W^{1,p}}

\def\uo{{u^0}}
\def\ui{{u^1}}
\def\uor{{u^{0,r}} }
\def\uir{{u^{1,r}}}
\def\ur{{u^r}} 
\def\uoi{{u^{0,1}}}
\def\uoir{{u^{0,1,r}}}

\def\ev{\, e \! v}
\def\timesi{\underset{i\in\zz}{\times}}
\def\liz#1{\ell^\infty(\zz;#1)}
\def\CR{\mathcal{S}}
\def\SR{\Sigma}

\newcounter{rind}
\newcounter{cind}
\def\indr{{ \addtocounter{rind}{1} \therind{} }}
\def\indc{{ \addtocounter{cind}{1} \thecind{} }}

\begin{document}

\title{Complex surfaces and interpolation on pseudo-holomorphic cylinders}
\author{Antoine Gournay}
\date{~} 

\maketitle

\begin{center}
Max Planck Institut für Mathematik, \\
 vivatsgasse 7,\\
 53111 Bonn, \\
Germany \\
\scriptsize{gournay@mpim-bonn.mpg.de}
\end{center}

\medskip

\begin{abstract}

Gromov has shown how to construct holomorphic maps of the plane to a complex manifold with prescribed values on a lattice. In the present paper, a similar interpolation theorem for pseudo-holomorphic maps from the cylinder $\CR$ to an almost-complex manifold $(M,J)$ is proved. Properties of the space of pseudo-holomorphic maps from $\CR$ to $(M,J)$ are derived, in particular a lower bound for the size of this space (in terms of mean dimension). When $M$ is a moduli space of curves, this gives a construction of non-compact complex surfaces. The methods involve an infinite number of surgeries. On the way, a refinement of the  gluing process for two pseudo-holomorphic curves is obtained, establishing the geometric behavior of two glued pseudo-holomorphic curves.

\end{abstract}

\section{Introduction} \label{intro}
The motivation for the present article is to construct a huge family of complex surfaces (and measure it). A convenient way to do so would be to realize them as fibrations, and consequently look for holomorphics maps from a non-compact Riemannian surface to the (compactified) space of genus $g$ curves $\overline{\jo{M}}_g$. The theorem of Gromov unfortunately does not cover this case, nor do essentially linear methods. 
\par We are thus lead to look at different methods, namely, those of gluing. The interpolation will be for maps from the cylinder $\CR = \cc / \zz$ to an almost-complex manifold $(M,J)$ given that there are pseudo-holomorphic curves $\cp^1 \to M$ which intersect in a cyclical fashion, that $J$ is regular (in the sense of definition \ref{regtransi}), and that there is a broader family of pseudo-holomorphic curves covering a neighborhood of one of the curves in the cycle (see section \ref{hypint}, in the setting of complex algebraic varieties this is called ``free'' curve). 
Some consequences of this result will then be explored, namely that the maps obtained actually form a very rich family of (unparametrized) cylinders.
\par The interpolation theorem requiring some technical preliminaries for a precise statement, we will only give the following rough result; see \ref{hypint} for all the details.
\begin{theorem}\label{inter}
Suppose $(M,J)$ is an almost complex manifold and that $J$ is regular (in the sense of definition \ref{regtransi}) of class $C^2$. Suppose there is a sequence of $J$-holomorphic curves $u^k: \cp^1 \to M$ where $k=1,2,\ldots N$ that intersect each other cyclically. Suppose further that $u^1$ belongs to a family of pseudo-holomorphic curves covering a neighborhood of $u^1(z_*)$ for some $z_* \in \cp^1$ (see section \ref{hypint}).
Then there is a family of pseudo-holomorphic map $v:\CR \to M$ such that $v$ is close to this cyclic sequence of curves and the values of $v$ on some lattice $z_0 + iN\zz \subset \CR = \cc / \zz$ can be prescribed to be any point in $U$. \par Furthermore, if $M$ is of (real) dimension greater than $4$, that the $u^k$ do not intersect tangentially and do not possess other intersections then $v^{-1}(v(z))$ is contained in some neighborhood of $z+iN\zz$.
\end{theorem}
\par Let us return to the introductory topic. Among the numerous references concerning uniruled varieties, the reader can look at \cite[Chapter 4]{Deb}. The conditions on the maps in this theorem can be reinterpreted (when $M$ is a complex algebraic variety) as the existence of a cycle of curves, one of which should be free (see \cite[Definition 4.5]{Deb}; this is equivalent to $M$ being uniruled \cite[Definition 4.2 and Corollary 4.11]{Deb}). Furthermore, the regularity of $J$, which is required for the invertibility of the linearization of $\db_J$ at the curves $u^k$, is satisfied when there is a ``very free'' curve. It is known that moduli spaces of curves of genus $g$ are unirational when $g \leq 14$ and rationally connected (equivalent to the existence of a ``very free'' curve; see \cite[Definition 4.3 and Corollary 4.17]{Deb}) for $g \leq 15$. They are also uniruled for $g \leq 16$. Thus above result applies at least when $g \leq 15$, but not in genus $g \geq 24$ as the space is then of generic type (see \cite{Far} for a survey on the topic, and \cite{BCF} for results concerning curves with  marked points).
\par The rigidity result (proposition \ref{imcjh}) and the estimate on mean dimension (proposition \ref{dimpos}) will now be used to get some properties of the space of complex surface resulting from theorem \ref{inter}. Suppose $p:V \to \CR$ is a complex surface fibered over the cylinder obtained by applying  theorem \ref{inter} to $\overline{\jo{M}}_g$. Let $\SR$ be a compact curve and $j:\SR \to V$ a  holomorphic map. Then $p\circ j(\SR)$ is a analytic compact connected set, \ie a point. If $\SR$ is smooth of the same genus as a generic fiber then $j(\SR)$ is a fiber or $j$ is constant. Consequently, if $f:V \to V'$ is a holomorphic map, then $p' \circ f$ is constant on the fibers. It factorizes as a holomorphic map $h: \CR \to \CR$. Let $k$ and $k': \CR \to \overline{\jo{M}}_g$ be the classifying maps (with values in the compactified moduli space of genus $g$ curves). If $\dim f(V) >1$, then on $k^{-1}(\overline{\jo{M}}_g)$ (the generic fibers) $k \circ h = k'$, and by continuity this also hold on the singular fibers. By proposition \ref{imcjh} $g$ must be an isomorphism. Thus $V$ and $V'$ are isomorphic and the fibrations equivalent.
\begin{corollary}
Let $g \leq 15$. The spaces of (noncompact) complex surfaces given by fibration of genus $g$ curves over the cylinder (with one marked point) obtained by applying theorem \ref{inter} has positive mean dimension, and two such surfaces are (holomorphically) isomorphic only if they differ by an automorphism of the cylinder.
\end{corollary}
\par The interpolation result is achieved by gluing the cyclic sequence of curves together while satisfying the additional constraint given by the values we want to be prescribed. The proof will consists in constructing an approximate solution (that is a map that is almost pseudo-holomorphic and passes by the prescribed points) and then deforming it (using an implicit function theorem) to a truly pseudo-holomorphic map. Note that the last point of the theorem is not trivial as the implicit function theorem will blur things. Thus one has to make sure that close to the points where gluing occurs some injectivity is retained. 
\par Section \ref{chadd} will describe how to alter the construction described in \cite{mds1}. Given two pseudo-holomorphic curves, it is known (under proper assumptions on $J$) that there exists a family of pseudo-holomorphic curves that can be obtained by gluing them. However the behavior of these curves is vague at best. 
Here is the important improvement we make. 
\begin{theorem} \label{add}
  Let $(M,J)$ be an almost-complex manifold. Let $u^h : \SR \to M$, where $h \in \{0,1\}$, be two $J$-holomorphic curves such that $u^h(0)=m_0$, $\nr{\dd u^h}_{L^\infty} \leq C $, $J$ is regular in the sense of \cite[Definition 10.1.1]{mds1} and $D_{u^h}$ are surjective. If in a local chart $u^h(z) = a^h z + O(\abs{z}^2) $, then $\exists r_\therind$ such that $\forall r \leq r_\therind$, $\exists u$ a $J$-holomorphic curve such that in a local chart, 
\[
u(z) = a^0 z + a^1\frac{r^2}{z} + O(r^{1+\eps})
\]
for all $z \in A_{r^{4/3},r^{2/3}} =  \{z | r^{4/3} < \abs{z} < r^{2/3} \}$ and where $\eps \in ]0,\frac{1}{3}[$; $r_\therind$ and $c_0$ depend on $C$, $\eps$, $a^h$, the second derivatives of $u^h$, $J$ (up to its second derivatives) and on the norm of the inverse to $D_{u^h}$.
  \end{theorem}
Theorem \ref{add} says that if the vectors $a^0$ and $a^1$ represent the tangent plane of the curves at the intersection in some local chart of the point where the gluing occurs, then the glued curve (of parameter $r$) has the roughly the behavior $u^r(z) = a^0 z + a^1 r^2/z$ close to a ring of radius $r$. This is achieved using a more precise approximate solution to the glued curve. Technical difficulties arise (mainly in the inversion of the linearization of $\db_J$ at the approximate solution), but they can be avoided by modifying the almost-complex structure. 
\par Theorem \ref{add} is of particular interest when the $a^h$ are linearly independent over $\cc$ (which requires that $M$ be of real dimension at least 4). Indeed, then a strangling phenomenon can be shown to happen, see remark \ref{etrang}. These conditions are not required as such in the gluing procedure, but they are essential to \S{}\ref{ntfam} and \S{}\ref{p:simp}. 
\par Section \ref{recinf} will then explain how to pass from the gluing of two curves to the gluing of an infinite number of curves. To do so one only requires to consider an hybrid $\ell^\infty(L^p)$ norm which will preserve the qualities we need of standard $L^p$ norms while reducing the problem of an infinite number of gluing to a finite one.
\par Finally, section \ref{esphol} will show how to simultaneously solve the pseudo-holomorphic equation $\db_J u =0$ in addition to the constraints given by prescribing value at points. Interpolation enables to show that the family of maps from the cylinder to $M$ is very large. This is not so surprising as the existence of a pseudo-holomorphic map $\cp^1 \to M$ will give rise to many maps $\CR \to \cp^1 \to M$. It is important to mention that the maps are not obtained in such a trivial fashion. First, they are of bounded differential. Second, the different maps obtained by the interpolation theorem can be shown to have distinct images (see section \ref{ntfam}). Also, if the maps $u^k$ do not have additional intersections an appropriate choice of parameters is sufficient to ensure that the map does not factor through an holomorphic map $\CR \to \SR$ (see section \ref{p:simp}). Lastly, as in \cite[{\S}3.3]{Gro}, one can show that the family of maps given by the interpolation theorem is of positive mean dimension (see section \ref{s:mdm}). 


\section{Summing $J$-holomorphic curves}\label{chadd}

The first aspect of the theorem we shall prove is the good control of the behavior of the resulting cylinders around the gluing points. But before we move to an infinite number of gluing, we shall at first do so with only two. This is a refinement of the theorem described by McDuff-Salamon in \cite{mds1} on the possibility of gluing two curves, \ie to find a family of curves whose images are close to the union of the images of two curves meeting at $m_0$. In short, it allows us to show that, given two $J$-holomorphic curves $u^0$ and $u^1$ that intersect but are not tangent at a point $m_0$ there exists a family of curves whose image are close to the union of the images of $u^0$ and $u^1$, and whose strangling close to $m_0$ are different.
\par This section describes how to modify this gluing so as to obtain that (in local charts near $m_0$ and $0 \in \cp^1$) in a ring of radius $r$ (around $0 \in \cp^1$) a local expansion would be of the form: $a^0z+ a^1\frac{r^2}{z} + O(r^{1+\eps})$, where $a^0$ and $a^1$ are the tangents to the curves at $m_0$, and $\eps \in ]0, 1/3[$. When $a^0$ and $a^1$ are linearly independent over $\cc$, this information will be used later in section \ref{esphol} to insure that the intersection of a ball of radius $O(r^{1+\eps})$ with the image of the map gives only discs when non-empty. The method is very close to that of \cite[{\S}10]{mds1}, which itself parallels \cite[{\S}7.2]{DK}.
\par Throughout this section $\SR$ will denote a Riemann surface (our interest is restricted to $\cp^1$) and $(M,J)$ will be an almost complex manifold of real dimension at least $4$. The almost complex structure $J$ will be assumed regular in the sense of \cite[Definition 10.1.1]{mds1} for the two curves considered and of class at least $C^2$. In particular, elliptic regularity insures that $J$-holomorphic maps will be at least $C^2$. 

\subsection{Definitions and description of the gluing map}
As we are concerned with local expansions, let us look at the local behavior of a $J$-holomorphic map. Let $a \in \rr^{2n}$ and $z \in \cc$, the product $a z$ means $z a = (x+ i y) a = xa + y J_0 a$, where $J_0:= \sm22{0}{-\un}{\un}{0}$. Note that the local charts will be chosen so that at $0 \in \rr^{2n}$ the almost-complex structure induced by $J$ (which will still be denoted by $J$) will be the usual complex structure, \ie $J(0)=J_0$. Another convention is that the evaluation of $J$ at a point $m$ will be written $J_m$; note that the confusion that could arise between the usual structure and the evaluation of $J$ at zero in a map is not to be worried about as, by choice of local charts, they will be equal. We start by this well-known lemma (e.g. \cite[Proposition 3]{Sik} where the proof is done under the weaker assumption that $J$ is Lipschitz).
  \begin{lemma} \label{devloc}
  Let $J$ be an almost complex structure on $\rr^{2n}$ such that $J(0)=J_0:= \sm22{0}{-\un}{\un}{0}$. Let $u:\cc \to \rr^{2n}$ be a $J$-holomorphic curve such that $u(0)=0$. Then $\exists a \in \rr^{2n}$ such that $u(z) = az + O(\abs{z}^2)$, for $\abs{z}$ small enough. 
  \end{lemma}
  \begin{proof}
  The notation $(\db_{J_g} f)(z) = \dd f (z) + J_{g(z)} \circ \dd f (z) \circ j $ will be used to insist on the point at which $J$ is evaluated. The first step is to remark that 
\begin{equation} \label{ast1}
  \db_J g = \db_{J''} g + (J_g -J''_g) J'_g (\del_{J'}-\db_{J'})g,
\end{equation}  
for any two complex structures $J'$ and $J''$, and where $\del_J = \db_{-J}$. On the other hand, write $u(z) = \sum_{k,l} z^k \zb^l a_{k,l} + O(\abs{z}^3)$, where $k,l \in \{0,1,2\}^2 \setminus \{0\}^2$, $a_{k,l} \in \rr^{2n}$ and $(s+it)a = a s + (J_0 a) t$. It appears, by choosing $J'=J''=J_0$ in \eqref{ast1} or by looking directly at the expression in local coordinates, that $\db_J u =0$ if and only if
\begin{equation}  \label{dbudev}
    \somme{k,l}{} a_{k,l} l z^k \zb^{l-1} + O(\abs{z}^2)+ (J_{u}-J_0) J_0 \left( \somme{k,l}{} a_{k,l} k z^{k-1} \zb^l- a_{k,l} l z^k \zb^{l-1} + O(\abs{z}^2) \right)  =0.
\end{equation}  
Furthermore, the coefficients $(c_1,c_2)$ of the matrix of $J$ can be expanded: 
\[
(J_{\vec{x}})_{c_1, c_2} =(J_0)_{c_1, c_2} + \somme{\vec{k}}{} b_{c_1,c_2,\vec{k}}(\vec{x})^{\vec{k}} +O(\abs{\vec{x}}^3), 
\]
where $\vec{k} \in \{0,1,2\}^{2n} \setminus \{0\}^{2n}$. consequently, there is only one term of order $0$ in (\ref{dbudev}): $a_{0,1}$. If $u$ is $J$-holomorphic, it's local expansion must be of the form $u(z) = a_{1,0}z + O(\abs{z}^2)$.
  \end{proof}
\par Before we proceed to the proof of theorem \ref{add}, let us note that the assumptions are more restrictive than in the gluing procedure of \cite[{\S}10]{mds1} where curves whose differential at $m_0$ is $0$ can be glued. For the remainder of this section, we shall assume that $a^0$ and $a^1$ are linearly independent (over $\cc$; whence the condition $\dim_\rr M \geq 4$). This assumption is actually not crucial to realize the gluing (though it makes things slightly simpler, see lemma \ref{mindu}), but is required in order to show that the strangling of the different curves is different (see remark \ref{etrang}) and that the resulting curve does not actually pass by $m_0$. 
\par The behavior of the ``summed'' curve is however more precise. Indeed in \cite[{\S}10]{mds1} the curve obtained by gluing is a perturbation of a curve which is constant in a ring; this leads to a curve whose behavior in the given ring is $u(z) = O(r)$. The price to pay to obtain a more precise behavior is that the approximate solution is no longer constant in a ring. When the approximate solution is constant in a ring $A$, the almost-complex structure $J$ is also constant for $z \in A$. Section \ref{jpc} describes how to modify the structure $J$ so as to make it constant near the point of intersection, thus allowing to avoid the difficulty that arises.
\par The main ingredient in the proof remains the implicit function theorem of \cite[{\S}3.5]{mds1}; recall that 
\begin{equation}  \label{cpdef}
s_p:= \sup_{0 \neq f \in C^\infty(\SR)} \fr{\nr{f}_{L^\infty}}{\nwup{f}}
\end{equation}
is the constant of the Sobolev embedding $\wup(\SR,\rr) \inj L^\infty(\SR,\rr)$, which is finite for $p> \dim \SR =2$ in our case (\cf \cite[{\S}6.7]{GT}).
\begin{e-proposition}\label{til} (see \cite[Theorem 3.5.2]{mds1})
Let $\SR$ be a complex manifold of dimension $1$, let $p>2$. $\forall c_\thecind$, $\exists \delta >0$ such that for all volume forms $\dd \textrm{vol}_\Sigma$ on $\Sigma$, all $u \in \wup(\Sigma,M)$, all $\xi_0 \in \wup(\SR,u^*\tg M)$, and all $Q_u: L^p(\SR, \Lambda^{0,1} \otimes_J u^*\tg M) \to \wup(\SR,u^*\tg M)$ satisfying
\[ 
  \begin{array}{ccc}
    s_p(\dd \textrm{vol}_\Sigma) \leq c_\thecind, & \nr{\dd u}_{L^p} \leq c_\thecind, & \nwup{\xi_0} \leq \frac{\delta}{8}, \\
 \nr{ \db_J(\exp_u(\xi_0))}_{L^p} \leq \frac{\delta}{4 c_\thecind}, & D_u Q_u = \un, & \nr{Q_u} \leq c_\thecind,
  \end{array}
\]
there exists an unique $\xi$ such that
\[ 
  \begin{array}[c]{ccc}
    \db_J(\exp_u(\xi_0+\xi))=0, & \nwup{\xi+\xi_0} \leq \delta, & \nwup{\xi} \leq 2 c_\thecind \nr{\db_J(\exp_u(\xi_0))}_{L^p.}
  \end{array}
\]
\end{e-proposition}
The proof of this theorem is a consequence of the implicit function theorem (\cf \cite[Proposition A.3.4]{mds1}), and we refer to \cite[{\S}3.5]{mds1} for the proof. A bound on the second derivative of $\jo{F}_u$ (\cf \cite[{\S}A.3]{mds1}) is required for it to hold. 
\par We start by constructing a family of curves $u^r$ whose local expansion is as required, which satisfy the conditions of the above theorem ($\xi_0$ will be $\equiv 0$) and whose $\db_J$ is of the order of $O(r^{1+\eps})$. Then, the $\xi$ obtained (the perturbation of $u^r$ needed to obtain a true solution) will be bounded in $L^\infty$ (since it is bounded in $\wup$) by $O(r^{1+\eps})$. 
\par Before we describe these maps $u^r$, we have to define cutoff functions which will be very useful. They will be denoted by $\beta$. The definition will not vary much, and, much like the following lemmas, is well-established; see \cite{DK} or \cite{mds1}.
\begin{e-definition} \label{betadef}
Let $\beta_{\delta,\eps} :\rr^2 \to \rr$ be the function defined by:
\[  
\beta_{\delta,\eps}(z)= \left\{
\begin{array}{llrcll}
    1                                                & \textrm{if}&          &\abs{z}& < \delta \\
    \fr{\ln \eps -\ln \abs{z}}{\ln \eps -\ln \delta} & \textrm{if}& \delta < &\abs{z}& < \eps \\
    0                                                & \textrm{if}& \eps   < &\abs{z}&        &.
\end{array} \right.
\]
\end{e-definition}
This cutoff function has many useful properties, as can be seen in the following two lemmas:
\begin{lemma} \label{l2dbeta}
$ {\displaystyle\vs \int_{\rr^2} }  \abs{ \nabla  \beta_{\delta,\eps} }^2 = \fr{2 \pi}{\ln (\eps/\delta)}$.
\end{lemma}
In particular, this first lemma shows that this family contains a limit case of the Sobolev embedding. Indeed, for fixed $\eps$ and if $\delta \to 0$, the function obtained is in $W^{1,2}$, but not in $L^\infty$. The second lemma is also true when $p<2$ without even needing to assume that $\xi(0)=0$.
\begin{lemma} \label{dbetaxi}
Let $\beta$ be as in definition \ref{betadef}, let $\xi \in W^{1,p}(B_{\eps})$ where $p>2$ be such that $\xi(0)=0$, then $\exists s_H$ such that: $\nr{(\nabla \beta) \cdot \xi}_{L^p(A_{\delta,\eps})}  \leq \fr{ (2 \pi)^{1/p}  s_H}{\ln(\eps/\delta)^{1-1/p}} \nr{\xi}_{W^{1,p}(B_{\eps})}$.
\end{lemma}
\indent To find a $J$-holomorphic curve with the desired local behavior, an intuitive idea would be to add up the local expansions of two curves, namely $\uo(z)$ and $\ui(r^2/z)$, when $\abs{z}$ is close to $r$ and to get back to either map outside and inside the ring. Addition does not exist in manifolds, thus it is necessary to choose local charts in order to achieve this. In the resulting formula, maps should be seen as functions from (an open set of) $\cc$ to (an open set of) $\rr^{2n}$. The family of maps $u^r$ will be defined as follows:
\begin{equation} \label{ur}
  u^r (z) = \left\{ 
    \begin{array}{llrcl}
      \ui(\frac{r^2}{z})                             &\textrm{if}&               &\abs{z}&<  r^{2-\gamma} \\
      \beta(\frac{r^2}{z}) \uo(z)+ \ui(\frac{r^2}{z})&\textrm{if}& r^{2-\gamma} <&\abs{z}&< r^{2-\alpha} \\
      \uo(z)+\ui(\frac{r^2}{z})                      &\textrm{if}& r^{2-\alpha} <&\abs{z}&< r^\alpha \\
      \uo(z)+\beta(z) \ui(\frac{r^2}{z})             &\textrm{if}& r^\alpha     <&\abs{z}&< r^\gamma \\
      \uo(z)                                         &\textrm{if}& r^\gamma     <&\abs{z}&
    \end{array} \right.
\end{equation}
where $0<\gamma < \alpha <1$, and the cutoff function is $\beta= \beta_{r^\alpha,r^\gamma}$ (see definition \ref{betadef}).

\subsection{Metrics and estimates} \label{smet}

Before we can estimate the norms of $\dd u^r$ and of $\db_J u^r$ (in $L^p$), we need to specify the metric on the domain $\SR$. When $\abs{z}> r$ the curve defined by $\ur$ will be close to $\uo$, and when $\abs{z}< r$, $\ur(r^2/z)$ resembles $\ui(z)$. These two subsets of the domain will play a similar role; it is natural to give them equal weights (at the domain). Intuitively, this also avoids the norm of the differential becoming large by giving to regions whose energy is of the same magnitude equal weights in the domain. This metric will be the usual (Fubini-Study) metric when $\abs{z}> r$, and the one induced by $z \mapsto \frac{r^2}{z}$ when $\abs{z} <r$ (see if necessary figure in the version available on the author's website). More precisely, the metric will be $g^r := (\theta^r)^{-2}(\dd s^2 + \dd t^2)$, where 
\[  
\theta^r(z) = \left\{
  \begin{array}{ll}
    r^2 + \abs{z}^2/r^2 & \textrm{si } \abs{z}<r \\
    1+ \abs{z}^2        & \textrm{si } \abs{z}>r
  \end{array} \right.
\]
\indent It might seem necessary to work with norms that take into account the two distinct regions, but since the situation is symmetric, estimates valid on a region will hold on the other. A more precise discussion can be found in \cite[{\S}10.3]{mds1}. We will only note that the volume remains bounded $Vol(\SR) \leq 2\pi$. The next lemma, taken from  \cite[{\S}10.3]{mds1} says that Sobolev constant behaves similarly.
\begin{lemma} \label{csob}
  The constant $s_p$ (\cf \eqref{cpdef}) for the metric $g^r$ remains bounded independently of $r$.
\end{lemma} 
\indent It is now possible to evaluate the $L^p$ norms of $\dd \ur$ and $\db_J u^r$ in order to satisfy the assumptions of proposition \ref{til}. Our starting point is to bound the norm of powers of $z$: 
\begin{lemma}\label{lpz} 
Let $r >0$, $l,l' \geq 1$, $0\leq \delta <\eps \leq 1$, and $\nr{\cdot}_{L^p(A_{r^\eps,r^\delta})}$ denote the $L^p$ norm restricted to the ring $A_{r^\eps,r^\delta}  = \{z| r^\eps < \abs{z} < r^\delta\} $. Then
\[ 
  \begin{array}{lll}
    \nr{z^l}_{L^p( A_{r^\eps,r^\delta})}
          &= \left( \frac{2 \pi (1-r^{(\eps-\delta) (2+lp) })}{2+lp} \right)^{1/p}r^{\delta(l+2/p)} &\sim K_{\eps,\delta,p,l} r^{\delta(l+2/p)} \\
    \nr{\frac{r^{l'}}{z^l}}_{L^p,(A_{r^\eps,r^\delta})}
          &= \left( \frac{2 \pi (1-r^{(\eps-\delta) (lp-2) })}{lp-2} \right)^{1/p}r^{l'+\eps(-l+2/p)} &\sim K'_{\eps,\delta,p,l} r^{l'+\eps(-l+2/p)}\\
  \end{array}
\]
where $K_{\eps,\delta,p,l}$ and $K'_{\eps,\delta,p,l}$ are the limits as $r \to 0$ of the terms before the powers of $r$.
\end{lemma}
\begin{proof}
It is a direct calculation, valid for $l \neq -2/p$:
\[
  \begin{array}{ll}
   \nr{z^l}_{L^p(A_{r^\eps,r^\delta})}^p
        &= \bint{A_{r^\eps,r^\delta}}{} \rho^{pl+1}  \dd \rho \dd \theta \\
        &= 2 \pi \left[ \fr{\rho^{2+lp}}{2+lp} \right]_{r^\eps}^{r^\delta}  \\
        &= \left( \frac{2 \pi (1-r^{(\eps-\delta) (2+lp) })}{2+lp} \right) r^{\delta(lp+2)} \\
  \end{array}
\]
A simple manipulation of this equality gives the second estimation.
\end{proof}
\begin{lemma} \label{lpdu} Let $r_\indr$ be such that $| \ln (r_\therind^{\alpha-\gamma} ) |^{-1}<1$, then
\[ 
\forall r<r_\therind, \nr{\dd u^r}_{L^p} \leq \nr{\dd \uo}_{L^p}+\nr{\dd \ui}_{L^p}+ c_\indc r^{2/p}, 
\]
where $c_\thecind = (4 K'_{\alpha,\gamma,p,2}r^{2(1-\alpha)}+ 2 K'_{1,\alpha,p,2})C $ and $C \geq \max(\nr{\dd \uo}_{L^\infty} , \nr{\dd \ui}_{L^\infty})$.
\end{lemma} 
\begin{proof}
In the region $r<\abs{z}<r^\alpha$ this is a simple assertion: 
\[
\nr{\dd u^r}_{L^p(A_{r,r^\alpha})} \leq \nr{\dd u^0}_{L^p(A_{r,r^\alpha})} + \nr{\frac{r^2}{z^2}}_{L^p(A_{r,r^\alpha})} \nr{\dd \ui}_{C^0} \leq \nr{\dd u^0}_{L^p(A_{r,r^\alpha})}  + C K'_{1,\alpha,p,2} r^{2/p}.
\]  
\indent Whereas when $r^\gamma < \abs{z}$, it is trivial since $\dd u^r = \dd u^0$. On $A_{r^\alpha,r^\gamma}$, a choice of a local chart and a local expansion for $\ui$ is needed: if $r$ is small, then $\frac{r^2}{z}$ is also small in the given region, $r^{2-\alpha} >\abs{\frac{r^2}{z}}>r^{2-\gamma}$. Indeed, 
\[
\nr{\dd u^r}_{L^p(A_{r^\alpha,r^\gamma})} \leq \nr{\dd u^0}_{L^p(A_{r^\alpha,r^\gamma})} + \nr{\dd \bigg(\beta(\abs{z}) \ui(\frac{r^2}{z}) \bigg)}_{L^p(A_{r^\alpha,r^\gamma})},
\]
and since $\abs{\ui (z)} \leq C \abs{z}$, the second term can be written as 
\[
  \begin{array}{rl}
    \nr{\dd (\beta(z) \ui(\frac{r^2}{z}))}_{L^p} 
     &\leq \nr{\frac{\ui(\frac{r^2}{z})}{\abs{z} \ln{r^{\alpha-\gamma}}}}_{L^p} + \nr{ \dd (\ui(\frac{r^2}{z})) }_{L^p} \\
     &\leq \nr{C \frac{r^2}{z^2}  }_{L^p} \abs{\ln{r^{\alpha-\gamma}}}^{-1} +  \nr{\dd \ui}_{C^0} \nr{\frac{r^2}{z^2}}_{L^p}. \\
  \end{array}
\]
As the terms appearing are of the form $\frac{r^2}{z^2}$, and using lemma \ref{lpz}, 
\[
\nr{\dd (\beta(z) \ui\bigg(\frac{r^2}{z}\bigg))}_{L^p} 
     \leq C K'_{\alpha,\gamma,p,2} (1+\abs{\ln{r^{\alpha-\gamma}}}^{-1}) r^{(1-\alpha)(2-2/p)}r^{2/p}.
\]
The contribution of the region $r^\alpha <\abs{z}<r^\gamma$ to $\nr{\dd u^r}_{L^p}$ tends to zero $0$ as $r \to 0$ faster than $r^{2/p}$. The final result follows from the symmetry which yields the same conclusion on the region where $\abs{z} <r$.
\end{proof}
\indent Our goal being to give a local expansion at order $1$, we have to show that the $L^p$ norm of $\db_J u^r$ (which bounds the $\wup$ norm and consequently the $L^\infty$ norm of the perturbation necessary to obtain a true solution) is $O(r^{1+\eps})$ when $z$ is of norm close to $r$.
\begin{lemma} \label{orddb}
Take $2<p<4$. Let $\alpha \in ]\frac{p}{p+2}, \frac{p}{2p-2}[ $ then there exists positive numbers $\eps < \min (\alpha(1+2/p)-1, 1-\alpha(2-2/p) )$, $r_\indr$ and $c_\indc$ (both depend on the second derivatives of $\uo$ and $\ui$, and on the product of the derivatives of $J$ with $C$) such that, $\forall r<r_\therind, \pnr{\db_Ju^r}_{L^p} \leq c_\thecind r^{1+\eps}$. 
\end{lemma}
\begin{proof}
Since the situation is symmetric, we will only be concerned with the part where $r<\abs{z}$. We split this region again, as the definition $u^r$ varies.
\[
\begin{array}{lll} 
  \pnr{\db_J u^r}_{L^p(\{z | \abs{z}>r \})}^p 
        &= \pnr{\db_J u^r }_{L^p(A_{r,r^\alpha})}^p 
        &+ \pnr{\db_J u^r }_{L^p(A_{r^\alpha , r^\gamma})}^p. \\
\end{array}
\]
The region where $\abs{z}>r^\gamma$ does not contribute in the equality above since $u^r=u^0$ is $J$-holomorphic. For the other domains, $J_w$ will be seen as a matrix valued map using a local chart. Again, the notation $(\db_{J_g} f)(z) = \dd f (z) + J_{g(z)} \circ \dd f (z) \circ j $ will be used to emphasize the point at which $J$ is evaluated. With this understood,
\[
\begin{array}{rl}
\nr{\db_J u^r }_{L^p(A_{r^\alpha , r^\gamma})} 
    &= \nr{\db_J u^r - \db_J u^0}_{L^p(A_{r^\alpha , r^\gamma})} \\
    &= \nr{\db_{J_\ur} ( u^r - u^0) +(\db_{J_\ur} - \db_{J_\uo}) u^0}_{L^p(A_{r^\alpha , r^\gamma})} \\
    &= \nr{\db_{J_\ur} \big(\beta(z) \ui(\frac{r^2}{z}) \big) + (J_\ur-J_\uo)\dd u^0 \circ j }_{L^p(A_{r^\alpha , r^\gamma})} \\
    &\leq 
\nr{ \dd \big( \beta(z) \ui(\frac{r^2}{z}) \big) }_{L^p(A_{r^\alpha , r^\gamma})}
+ \nr{J_.}_{C^1} \nr{\dd \uo}_{C^0} \nr{\ui(\frac{r^2}{z})}_{L^p(A_{r^\alpha , r^\gamma})}. \\ 
\end{array}
\]
The bounds obtained in lemma \ref{lpdu} and the norms computed in lemma \ref{lpz} yield the following upper bound:
\[
\pnr{\db_J u^r }_{L^p(A_{r^\alpha , r^\gamma})} \leq C K'_{\alpha,\gamma,p,2} (1+\abs{\ln r^{\alpha-\gamma}}) r^{2-\alpha(2-2/p)}+\nr{J_\cdot}_{C^1} C^2 K'_{\alpha,\gamma,p,1} r^{2-\alpha(1-2/p)}
\]
\indent In order to factorize $r^{1+\eps}$ when $l=1$ or $2$, one must have that $2 - \alpha (l-2/p) >1 \ssi \alpha < \frac{p}{lp-2}$. This condition is only restrictive for $l=2$.
\par To evaluate the other part, we proceed as in lemma \ref{devloc}. Upon noticing that the local expansion of $\db_J u$ implies $\abs{\db_J u} \leq \abs{\frac{\del u}{\del s} + J(u) \frac{\del u}{\del t}}$, and that the $u^h$ can be written as
\[
u^h(z) = a^hz + \sum_{k,l} z^k \zb^l a^h_{k,l} +O(\abs{z}^3) \qquad \textrm{ where }k,l \in \{0,1,2\}^2, k+l \geq 2
\]
when $r<\abs{z}<r^\alpha$, the following bound (it can also be seen using \eqref{ast1}) appears 
\[
\begin{array}{rl}
\big| \db_J u^r \big|
     &\leq \bigg| a^0_{1,1} z + 2 a^0_{0,2}\zb - a^\infty_{1,1} \frac{r^4}{z \zb^2} - 2a^\infty_{0,2} \frac{r^4}{\zb^3} +O(\abs{z}^2)+O( \frac{r^6}{\abs{z}^4}) \\
     &\qquad \qquad +(J_\ur-J_0) J_0 (a^0_{1,0} + a^\infty_{1,0} \frac{r^2}{z^2} +O(\abs{z}) +O(\frac{r^4}{\abs{z}^3} ) ) \bigg|.
\end{array}
\]
Thus, the factors $\abs{z}$, $ \frac{r^2}{\abs{z}}$ and $ \frac{r^4}{\abs{z}^3}$ could endanger our goal, as an expansion of $(J_\ur-J_0)$ shows. It also means that our bounds depend on the second derivatives of the $u^h$, or on a product of the first derivatives of $J$ and $u^h$. The $L^p$ norm will be made of terms in
\[
  \begin{array}{ll}
    \nr{z^l}_{L^p(A_{r,r^\alpha})}     
              &\sim K_{1,\alpha,p,l} r^{\alpha(l+2/p)} \\
    \nr{\frac{r^{l'}}{z^l}}_{L^p(A_{r,r^\alpha})}
              &\sim K'_{1,\alpha,p,l} r^{l'-l+2/p} \\
  \end{array}
\]
with $l'>l \geq 1$, as computed in lemma \ref{lpz}. This raises a new condition on $\alpha$: $\alpha(l+2/p) >1 \ssi \alpha> \frac{p}{lp+2}$. Consequently $\pnr{\db_J u^r}_{L^p} \leq K r^{1+\eps}$ under the condition that $\alpha \in ]\frac{p}{p+2}, \frac{p}{2p-2}[$, which is only possible if $p<4$. 
\end{proof}
\begin{remark} \label{chalpha}
  For any $p \in ]2,4[$, there is an optimal choice of $\alpha$. Indeed, if $\alpha = \frac{2}{3}$ then $\eps < \frac{1}{3}(\frac{4}{p}-1)$. Taking $p$ close to $2$, enables $\eps$ to be close to $1/3$. Theorem \ref{add} is obtained with this choice of $\alpha$. However, it is not possible to take $p \to 2$ as some constants, \eg $s_p$, depend on $p$. The choice of $\gamma$ is quite secondary, \eg one could choose $\gamma =\frac{5}{6}$.
\end{remark}

\subsection{Construction of the inverse $Q_{\ur}$} \label{sinv}
  In this section, we will make the somehow strong assumption that $J$ is constant in a neighborhood of $m_0 \in M$; the reason why such a simplification is possible is explained in {\S}\ref{jpc}. The whole gluing process is presented in its proper order in section \ref{sread}.
\par Before we apply the implicit function theorem, it is required to have a bounded inverse to the linearization of $\db_J$ at $u^r$, $D_\ur$. The existence of inverses for $D_{u^h}$ combined with the observation that two maps which are close enough will have close linearization, will enable the construction of this inverse. First let us show that if $u'$ is close to $u$ in the sense of $\wup$, then the operators $D_u$ and $D_{u'}$ are close. In order to identify their images, parallel transport is necessary. However, it does not affect significantly the following computation:
\begin{equation}\label{appu}
    \begin{array}{rl} 
\nr{D_u\xi - D_{u'} \xi}_{L^p} 
    & \leq \nr{(J_u-J_{u'})\nabla \xi}_{L^p} 
          + \frac{1}{2} \nr{J_u \nabla_\xi J_u (\dd u -\dd u') }_{L^p} 
  \\ & \qquad + \frac{1}{2} \nr{(J_u \nabla_\xi J_u - J_{u'} \nabla_\xi J_{u'}) \dd u' }_{L^p} \\
    & \leq \nr{J_.}_{C^1} \nr{u-u'}_{C^0} \nr{\nabla \xi}_{L^p} 
          + \frac{1}{2} \nr{J_u \nabla_\xi J_u}_{C^0} \nr{ \dd u -\dd u'}_{L^p} 
  \\ & \qquad + \frac{1}{2} \nr{J_. \nabla_\xi J_. }_{C^1} \nr{u-u'}_{C^0} \nr{\dd u' }_{L^p}\\
    &   \leq s_p \nr{J_.}_{C^1} \nwup{u-u'} \nwup{\xi} 
         + \frac{1}{2} \nr{J_u \nabla_\xi J_u}_{C^0} \nwup{ u - u'} 
  \\ & \qquad + \frac{1}{2} s_p \nr{J_. \nabla_\xi J_. }_{C^1} \nwup{u-u'} \nr{\dd u' }_{L^p} \\
    &   \leq s_p \nr{J_.}_{C^1} \nwup{u-u'} \nwup{\xi} 
  \\ & \qquad + \frac{1}{2} s_p \nr{J_u \nabla J_u}_{C^0} \nwup{\xi} \nwup{ u - u'} 
  \\ & \qquad \qquad + \frac{1}{2} s_p^2 \nr{J_. \nabla J_. }_{C^1} \nwup{\xi} \nwup{u-u'} \nr{\dd u' }_{L^p} \\
     &  \leq c_\indc (\nabla^2 J, \dd u', s_p )  \nwup{\xi} \nwup{u-u'} .
    \end{array}
\end{equation}  

\indent For the curves we are concerned with, proximity in $ \nwup{\cdot} $ will be insured as follows: $\dd (u^r- u^0)$ is zero when $\abs{z}> r^\gamma$, and it is of the order of $\frac{r^2}{z^2}$ when $r< \abs{z}< r^\gamma$, consequently $\nr{u^r-u^0}_{\wup(\{\abs{z}> r\})}$ is of the order of $r^{2/p}$. Thus, $D_\ur$ will be close to one of the $D_{u^h}$ inside or outside $\abs{z} =r$. 
\par To be more precise, it is necessary to introduce intermediate curves, denoted by $\uor$ and $\uir$. The first will be defined as follows
\[
  \uor (z) = \left\{ 
    \begin{array}{llrcl}
      \uo(z)                                          &\textrm{if}&               &\abs{z}&<  r^{2-\gamma} \\
      \uo(z)+ \beta(\frac{r^2}{z}) \ui(\frac{r^2}{z}) &\textrm{if}& r^{2-\gamma} <&\abs{z}&< r^{2-\alpha} \\
      \uo(z)+\ui(\frac{r^2}{z})                       &\textrm{if}& r^{2-\alpha} <&\abs{z}&< r^\alpha \\
      \uo(z)+\beta(\abs{z}) \ui(\frac{r^2}{z})        &\textrm{if}& r^\alpha     <&\abs{z}&< r^\gamma \\
      \uo(z)                                          &\textrm{if}& r^\gamma     <&\abs{z}
    \end{array} \right. ,
\]
and the second in an analogous manner. Since $\nwup{\uor-\uo} \to 0$ as $r \to 0$, the operator $D_\uor$ will be as close as required to $D_\uo$ and identical to $D_\ur$ when $\abs{z} > r$.
\par The two inverses $Q_{\uo}$ and $Q_{\ui}$ will be used to construct an inverse to $D_\ur$ whose bound is independent of $r$. First we introduce some notations. For $u:\SR \to M$, let $\wup_u = \wup(\SR, u^*\tg M)$, $L^p_{u} = L^p(\SR, \Lambda^{0,1}  \tg^*\SR \otimes_J u^*\tg M)$. Given $\uo,\ui: \SR \to M$, such that $\uo(0) = \ui(0)$, denote by
\[
\wup_{\uoi} := \left\{ (\xi^0,\xi^1) \in \wup_{\uo} \times \wup_{\ui} | \xi^0(0)=\xi^1(0)  \right\} .
\]
The assumption that $p>2$ is of importance, since $\wup$ sections need not be continuous if $p \leq 2$, and their evaluation at a point would not make sense.
\par Thanks to the regularity assumption made on $J$, the operator
  \[
  \begin{array}{rrll}
  D_{0,1} : & \wup_{\uoi}   & \to     & L^p_{u^0} \times L^p_{\ui} \\
            & (\xi^0,\xi^1) & \mapsto & (D_\uo \xi^0, D_{\ui} \xi^1)
  \end{array}
  \]
  is surjective (\cf \cite[{\S}10.5]{mds1}). Thus, $D_{0,1}$ possesses an inverse which depends continuously on the pair $(\uo, \ui)$ and satisfies an uniform bound as $(\uo, \ui)$ varies in $\jo{M}^*(C)$. This suffices for our use, but if one would like to stay in a case where ``surjectivity'' of the gluing map (\cf \cite[Theorem 10.1.2.iii]{mds1}) is possible, one needs to show that amongst all the inverses of $D_{0,1}$, the one which is orthogonal to the kernel also has bounded norm. Recall that surjectivity is the property that any $J$-holomorphic curve which is close to union of the curves $\uo$ and $\ui$ is in the image of the gluing map.
\par More precisely, if $\jo{W}_{\uoi} \subset \wup_{\uoi}$ is the ($L^2$) orthogonal to the kernel of $D_{0,1}$, then the restriction of this operator to $\jo{W}_{\uoi}$ is bijective and bounded. Its inverse will be denoted $Q_{0,\infty}$. It varies continuously with the pair $(\uo,\ui)$ and the bound is uniform as $\jo{M}^*(C)$ is compact. To make this explicit, an identification must be made between $\wup_{\uoi} $ and $\wup_{v^{0,1}}$ for pairs $(\uo, \ui)$ and $(v^0, v^1)$ sufficiently close (we will not do it here). 
\par Since the maps $\uor$ and $\uir$ are small $\wup$ deformations of $\uo$ and $\ui$, the space $\wup_{0,1}$ may be seen as a limit when $r \to 0$ of spaces $\wup_{0,1,r}$ corresponding to these slightly altered maps. The operator $D_{0,1,r}$ being a small perturbation of $D_{0,1}$ it will possess a right inverse. To prove surjectivity of the gluing map (as in \cite[Theorem 10.1.2.iii or Corollary 10.1.3]{mds1}, it is the inverse $Q_{0,1,r}$ whose image is $L^2$-orthogonal to the kernel of $D_{0,1,r} $ which must be chosen. A verification must be made to show that the bound on the norm of this operator is independent of $r$. This argument (\cite[Lemma 10.6.1]{mds1}) works without need of change in the situation in which we are (the kernel of $D_u$ is finite-dimensional).
\par Thanks to the operator $D_\uoir$, an approximate inverse $T_\ur : L^p_\ur \to \wup_\ur$ for $D_\ur$ will be obtained. Let $\eta \in L^p_\ur$. This 1-form will be cut along the circle $\abs{z}=r$ in two pieces $(\eta^0, \eta^1)$:
  \[
    \begin{array}{cc}
    \eta^0(z) = \left\{
      \begin{array}{ll}
       \eta(z) & \textrm{if } \abs{z} > r \\
       0       & \textrm{if } \abs{z} < r \\
      \end{array} \right. , 
       &  \eta^1(z) = \left\{
      \begin{array}{ll}
       r^2 \eta(r^2 z) & \textrm{if } \abs{z} < 1/r \\
       0       & \textrm{if } \abs{z} > 1/r \\        
      \end{array} \right. .
    \end{array}
  \]
\indent Since the $\eta^h$ are only in $L^p$, the discontinuity is not problematic. Now let $(\xi^0,\xi^1) = Q_{0,1,r}(\eta^0, \eta^1)$. It is worth stressing that $\xi^0(0)=\xi^1(0) =: \xi_{m_0} \in \tg_{m_0} M$. Let $\delta \in ]0,1[$. This choice is not of importance; it would suffice to take $\delta = \frac{1}{2}$. Let 
\[
\beta(z)= 1- \beta_{r^{1+\delta}, r}(z) = \left\{
\begin{array}{llrcl}
    1                                                      & \textrm{if}&          r  <& \abs{z} \\
    \fr{\ln \abs{z} - \ln(r^{1+\delta})}{ -\ln(r^\delta) } & \textrm{if}& r^{1+\delta} <& \abs{z} &< r \\
    0                                                      & \textrm{if}&              & \abs{z} &< r^{1+\delta}
\end{array} \right. .
\]
The approximate inverse is: $T_\ur \eta = \xi^r$, where
\begin{equation}  \label{xir}
 \xi^r(z) = \left\{
 \begin{array}{llrlc}
   \xi^0(z)                                                        &\textrm{if}& r^{1-\delta} <&\abs{z} \\
   \xi^0(z) +\beta(\frac{r^2}{z} )(\xi^1(\frac{z}{r^2}) -\xi_{m_0})&\textrm{if}&            r <&\abs{z}&< r^{1-\delta} \\
   \xi^0(z) + \xi^1(\frac{z}{r^2}) -\xi_{m_0}                      &\textrm{if}&            r =&\abs{z} \\
   \xi^1(\frac{z}{r^2}) +\beta(z) (\xi^0(z) -\xi_{m_0})            &\textrm{if}& r^{1+\delta} <&\abs{z}&< r\\
   \xi^1(\frac{z}{r^2})                                            &\textrm{if}&               &\abs{z}& <r^{1+\delta}\\
  \end{array} \right. 
\end{equation}
It remains to show that this is an approximate inverse as claimed, \ie $\nr{D_\ur \xi^r - \eta}_{L^p} \leq \eps \nr{\eta}_{L^p}$ for some $\eps \in [0,1[$. By construction the left-hand term is zero outside $ r^{1+\delta} < \abs{z} < r^{1-\delta}$. The assumption that the almost-complex structure is constant on that region will now be important. We restrict our attention, thanks to symmetry, to the piece $\abs{z} <r$. By definition $D_\ur \xi^1 ( \cdot/r^2) = \eta(\cdot)$. Hence, 
\begin{equation} \label{etar}
  \begin{array}{rl}
    D_\ur \xi^r - \eta &= D_\uor(\beta (\xi^0 - \xi_{m_0})) \\
                       &= \beta  D_\uor(\xi^0 - \xi_{m_0}) + (\xi^0 - \xi_{m_0}) \db \beta \\
                       &= (\xi^0 - \xi_{m_0}) \db \beta , 
  \end{array}
\end{equation}
 since $D_\uor \xi^0 =0$ when $\abs{z} < r$. It remains to bound this norm with respect to the metric (that depends on $r$). There will be a factor of $\theta^r(z)^{p-2}$ in the norm of the  1-forms, but \mbox{$\theta^r <\theta^1 \leq 2$}).
\begin{equation} \label{brnpinv}
   \begin{array}{rl}
    \nr{D_\ur \xi^r - \eta }_{L^p(B_r)} 
          & \leq 2^{1-2/p} \nr{D_\ur \xi^r - \eta }_{L^p(\abs{z} < r)} \\
          & \leq 2^{1-2/p} \nr{(\xi^0 - \xi_{m_0}) \db \beta }_{L^p(\abs{z} < r)} \\
          & \leq 2 \pi^{1/p} s_H \frac{\nwup{\xi^0 - \xi_{m_0}}}{\abs{\delta \ln r }^{1-1/p}} \\
          & \leq 2 \pi^{1/p} s_H s_p c_4 \frac{(\nr{\eta^0}_{L^p}+\nr{\eta^\infty}_{L^p})}{\abs{\delta \ln r }^{1-1/p}} \\
          & \leq 4 \pi^{1/p} s_H s_p c_4 \frac{\nr{\eta}_{L^p}}{\abs{\delta \ln r }^{1-1/p}} .
   \end{array}
\end{equation}
Lemma \ref{dbetaxi} is used to go from the 2$^{\textrm{nd}}$ to the 3$^{\textrm{rd}}$ line. The 4$^{\textrm{th}}$ line is obtained from the third using that $\xi_{m_0}$ is bounded by the $C^0$ norm (and thus by the $\wup$ norm) of $\xi^0$, and on the other hand that the $\wup$ norm of $\xi^0$ is bounded by a constant (coming from the norm of $Q_{0,1,r}$) multiplied by the $L^p$ norm of $\eta^0$.
\begin{lemma}\label{pinv}
Let $u^r$ be as defined in \eqref{ur} then $\forall \eps \in [0,1[, \exists c_\indc ,\exists r_\indr$ (which depend on $c_4$, $s_p$, and $s_H$), such that $\forall r \leq r_\therind, \exists T_\ur$ such that $\nr{D_\ur T_\ur -\un} \leq \frac{1}{2}$ and $\nr{T_\ur} \leq c_\thecind$.
\end{lemma}
\begin{proof}
The only part of the statement which was not proved in the above discussion is the one concerning the norm of $T_\ur$. It requires a bound on $\nwup{\xi^r}$ as a function of $\nr{\eta}_{L^p}$. Only the cutoff function requires care, the bound being otherwise found thanks to the bound on $Q_{0,1,r}$. However, $\nr{\nabla \xi^r}$ remains controlled exactly as in \eqref{brnpinv} thanks to lemma \ref{dbetaxi}.
\end{proof}
Thus the true inverse $Q_\ur$ will have the same image as $T_\ur$ and will be defined by:
\begin{equation}  \label{qur}
Q_\ur = T_\ur (D_\ur T_\ur)^{-1} = T_\ur \somme{k=0}{\infty} (\un - D_\ur T_\ur)^k
\end{equation}
It satisfies the relation: $D_\ur Q_\ur = \un$ and $\nr{Q_\ur} \leq 2 c_\thecind$ , where $c_\thecind$ comes from lemma \ref{pinv}.

\subsection{On the assumption that $J$ is constant near $m_0$} \label{jpc}
This section consists in noting that when $J$ is close to $J'$, the operator $D_u^J$ is close to $D_u^{J'}$ for certain $u$ (\eg $\uo, \ui$ and $\ur$). In order to speak of a difference between these two operators, we will see their images not as the space of $(0,1)$-forms taking value in $\tg M$ (since the definition of a $(0,1)$-form depends on the almost-complex structure) but as the space of $\tg M$-valued $1$-forms.
\[
\begin{array}{rl}
\pnr{D_u^J \xi - D_u^{J'} \xi}_{L^p} 
	&\leq \frac{1}{2} \nr{(J_u - J'_u) \nabla \xi}_{L^p} + \frac{1}{2} \nr{(J_u \nabla_\xi J_u - J'_u \nabla_\xi J'_u) \dd u }_{L^p} \\
	&\leq \frac{1}{2} \nr{\nabla \xi}_{L_p} \nr{J_u - J'_u}_{C^0} + \frac{1}{2} \nr{\dd u}_{C^0} \nr{\xi}_{C^0} \nr{J_u \nabla J_u - J'_u \nabla J'_u}_{L_p} \\
	&\leq c_\indc(\dd u) \nwup{\xi} (\nr{J_u - J'_u}_{C^0} + \nr{J_u \nabla J_u - J'_u \nabla J'_u}_{L_p}).
\end{array}
\]
\indent Thus, it is important to note that the dependence on the differential of $u$ will not be a problem for the maps we consider.
\begin{lemma} \label{c0du}
$\exists r_\indr (\alpha, \gamma)$ such that $\forall r < r_6$, $\nr{\dd u^r}_{C^0} \leq 2C$. 
\end{lemma}
\begin{proof}
This proof works in an analogous fashion as the bound of the $L^p$ norm of $\dd u^r$. When $r<\abs{z}<r^\alpha$ this is a simple thing to check: 
\[
\nr{\dd u^r}_{C^0} \leq \nr{\dd u^0}_{C^0} + \pnr{\tfrac{r^2}{z^2}}_{C^0} \nr{\dd \ui}_{C^0}
                   \leq 2C .
\]
If $r^\gamma < \abs{z}$, then $\dd u^r = \dd u^0$ so the conclusion is direct. Finally on $A_{r^\alpha,r^\gamma}$, the computation requires a local chart and a local expansion for $\ui$: if $r$ is small, then $\frac{r^2}{z}$ is also small on $A_{r^\alpha,r^\gamma}$: $r^{2-\alpha} >\abs{\frac{r^2}{z}}>r^{2-\gamma}$. Thus, $\nr{\dd u^r}_{C^0} \leq \nr{\dd u^0}_{C^0} + \nr{\dd (\beta(\abs{z}) \ui(\frac{r^2}{z}))}_{C^0}$, and since $\abs{\ui(w)} \leq C \abs{w} + O(\abs{w}^2)$, the second term can be written as 
\[ 
  \begin{array}{rl}
    \nr{\dd (\beta(z) \ui(\frac{r^2}{z}))}_{C^0} 
     &\leq \nr{\frac{\ui(\frac{r^2}{z})}{\abs{z} \ln{r^{\alpha-\gamma}}}}_{C^0} + \nr{ \dd (\ui(\frac{r^2}{z})) }_{C^0} \\
     & \leq \nr{C r^2/z^2 + O(r^4/z^3) }_{C^0} \abs{\ln{r^{\alpha-\gamma}}}^{-1}   +  \nr{\dd \ui}_{C^0} \nr{\frac{r^2}{z^2}}_{C^0} \\ 
     & \leq C \abs{r^{2-2\alpha}} (\abs{\ln{r^{\alpha-\gamma}}}^{-1} +1 )
  \end{array}
\] 
\indent The $C^0$ norm is bounded by the maximum of the bounds on each region: $\nr{\dd u^r}_{C^0} \leq \max(2C,C,C+o(1)) \leq 2C$ for all $r$ such that $\abs{r^{2-2\alpha}} (\abs{\ln{r^{\alpha-\gamma}}}^{-1} +1 ) <1$.
\end{proof}
\indent This lemma, together with lemma \ref{mindu}, allows us to choose $r$ arbitrarily small without changing the proximity of $D^J_\ur$ and $D^{J'}_\ur$. It is important to show that this proximity is valid for the whole family of curves considered. The property required of $J'$ is to be constant in a neighborhood of $m_0$. Consequently let us define for $R \in ]0,1[$ and for $\kappa \in \rr>0$,
\begin{equation}   \label{JR}
  J'(w) = \left\{ 
    \begin{array}{llrcl}
      J_0                  & \textrm{if} &                  & \abs{w} &< R(1-R^\kappa) \\
      J_{\beta(\abs{w})w}  & \textrm{if} &  R(1-R^\kappa) < & \abs{w} &< R \\
      J_w                  & \textrm{if} &       R <        & \abs{w} & \\
    \end{array} \right.
\end{equation}
where
\[
\beta(x)= \left\{
\begin{array}{llrcl}
  0                                                & \textrm{if}&               &   x < & R(1-R^\kappa) \\
  \fr{\ln x -\ln R(1-R^\kappa)}{-\ln (1-R^\kappa)} & \textrm{if}& R(1-R^\kappa) & < x < & R \\
  1                                                & \textrm{if}& R             & < x
\end{array} \right.
\]
Then $\nr{J-J'}_{C^0} \leq 2\nr{J-J_0}_{C^0(\{ \abs{w}<R \})} \leq O(R)$. Furthermore, since 
\[
\begin{array}{rl}
\abs{ \nabla (J_{\beta(\abs{w})w}) } 
      &\leq \abs{ (\nabla J)_{\beta(\abs{w})w} \bigg( \beta(\abs{w})- \fr{w}{\abs{w} \ln (1-R^\kappa)} \bigg) }  \\
      &\leq \abs{ (\nabla J)_{\beta(\abs{w})w}} \big(1 + \abs{\ln (1-R^\kappa)}^{-1}  \big),
\end{array}
\]
it is possible to obtain a rough bound for $\nr{J\nabla J - J' \nabla J'}_{L^p}$, if we suppose that $\pnr{\dd u} \geq d $ in $B_\rho(0)$, so that the preimage by $u$ of a small ball remains a small ball up to multiplication by a bounded factor. Let $R'$ be such that $B_{R'}(m_0) \cap \img u^h \subset u^h(B_\rho(0))$. In order to avoid cases where the map sends many subsets of $\cp^1$ to $B_{R'}(m_0)$ (for instance, is non injective), it is possible to introduce almost complex structures that depend on a point of the domain; we shall not go into such details. Thus, the preimage of $\abs{w} <R(1-R^\kappa)$ by $u$ is dilated by at most $k \propto \frac{1}{d}$, whence 
\[
\nr{J \nabla J}_{L^p(u^{-1}(\abs{w} <R(1-R^\kappa)) )} \leq \nr{J\nabla J}_{L^\infty} (k R^2(1-R^\kappa)^2)^{1/p},
\]
and similarly for the ring $R(1-R^\kappa) < \abs{w} < R$, the bound is 
\[
\nr{J \nabla J- J' \nabla J'}_{L^p(u^{-1}( R(1-R^\kappa) < \abs{w} < R ))} 
      \leq \nr{J\nabla J}_{L^\infty} \bigg(2+\frac{1}{\ln (1-R^\kappa)} \bigg) (k R^{2+\kappa} (2-R^\kappa))^{1/p}.
\]
 If $R^{2+\kappa(1-p)} \to 0$ (\eg if $\kappa = \frac{1}{p-1}$ and $R \to 0$), the operators associated to $J$ and $J'$ will be as close as needed.
\par There remains to check that the assumption on the lower bound on the differential holds for $\uo, \ui$ and $\ur$. For $\uo$ and $\ui$ it follows from the fact that $a^0$ and $a^1$ are not trivial. As for $\ur$, it is a consequence of their linear independence over $\cc$: let $\mu = \min_z (\abs{a^0 + z a^1}, \abs{a^0 z+ a^1})$, then, on $A_{r,r^\alpha}$, $\dd \ur = \dd \uo - \frac{r^2}{z^2}  \dd \ui$ has norm bounded from below by $\mu$. The cutoff function $\beta$ is not of importance since it is always multiplied by one of these linearly independent factors (in $u^h \nabla \beta$ or in $\beta \dd u^h$). Thus if the first order terms in local expansions are dominant, the same bounds hold on $A_{r^\alpha,r^\gamma}$. In short, we have proved the following lemma.
 \begin{lemma} \label{mindu}
$\exists r_\indr (a^0,a^1,\nabla^2u^h),d(a^0,a^1)$ such that $\forall z \in \{ z | \abs{z} \leq r_\therind \}, \abs{\dd u^h(z)} \geq d$, and $\forall r,z$ satisfying $\max(\abs{z},\abs{\frac{r^2}{z}} ) \leq r_\therind$, $\abs{ \dd \ur(z)} \geq d $. 
 \end{lemma}
Recall that this lemma and remark \ref{etrang} are the only place where the fact that the curves are not tangent at their intersection point is used. However, the use of this lemma is to get a bound on the size of the preimage of the region where the almost-complex structure $J$ is modified. This bound can be obtained without this assumption, but remark \ref{etrang} would no longer hold.
This section is summarized in the next proposition.
 \begin{e-proposition}
 Let $\uo$ and $\ui$ be as in the assumptions of theorem \ref{add}. Let $J'$ be the almost complex structure on $M$ which is constant on a neighborhood  $B_R(m_0)$ of $m_0$ defined in \eqref{JR}. $\forall u$ such that $z \in B_\rho \imp \abs{\dd u(z)} \geq d $, $\exists c_7(\nr{\dd u}_{L^\infty},\nr{\nabla J}_{L^\infty},\nr{J \nabla J}_{L^\infty},d)$ which makes the following true:  $\pnr{D^J_u - D^{J'}_u} \leq c_7 R^{[2+\kappa (1-p)]/p}$. 
\par In particular, the curves $\uo, \ui$, and all the $\ur$ (as defined in \eqref{ur}) for $r<r_\therind$ satisfy this condition, for the same constant $c_7$ since their differential is uniformly bounded when $r< \min (r_6,r_7)$.
 \end{e-proposition}
In order to justify the assumption that $J$ was constant in a neighborhood of $m_0$ made in {\S}\ref{sinv}, it suffices to construct this $J'$. In the end, the inverse of $D^{J'}_\ur$ obtained will be an approximate inverse to $D^{J}_\ur$. The independence of $J'$ with respect to $r$ is crucial for this new structure to be usable.

\subsection{Realizing the sum} \label{sread}

This section presents the proof of theorem \ref{add}; it is a matter of setting up the situation so that proposition \ref{til} can be applied. First, by assumption we are given two curves $u^0$ and $u^1$ whose tangents at $0$ are linearly independent (over $\cc$) and the linearized operators $D_u^J$ are surjective of bounded right inverses. Let $p \in ]2,4[$, let $u^r$ be the family of maps introduced in \eqref{ur}, with parameters $\alpha = \frac{2}{3}$ and $\gamma = \frac{5}{6}$ (as specified in remark \ref{chalpha}). Thanks to lemma \ref{lpdu}, if $r<r_1=e^{-6}$ 
\[
\nr{\dd u^r}_{L^p} \leq \nr{\dd u^0}_{L^p} +\nr{\dd u^1}_{L^p} + c_1 r^{2/p}.
\]
On the other hand, when $r<r_2(\nabla^2u^h,C \nabla J)$, lemma \ref{orddb} states that
\[
\pnr{\db_J u^r} \leq c_2 r^{1+\eps}.
\]
Before we can invoke proposition \ref{til}, we must show that there is a bounded (independently  
 of $r$) tight inverse to $D_{u^r}$. Two uniform bounds (for $r$ sufficiently small) are obtained by lemmas \ref{c0du} and \ref{mindu}. The first gives an upper bound to $\abs{\dd u^r}$ when $r< r_6$. The second gives a lower bound for the differentials when $r<r_7$ and $r^2/r_7 <\abs{z} < r_7$. Next, when $R$ is small enough so that $B_R(m_0) \cap \del u^h(B_\rho(0) ) = \vide$, the operators $D_{u^0}^J$, $D_{u^1}^J$ and $D_{u^r}^J$ are arbitrarily close (by choosing $R$ arbitrarily small) from the one defined by the structure $J'$ of \eqref{JR}. The difference between these operators is uniform for all choices of parameter $r$ smaller than $r_6, r_7$ and $\rho$. 
\par Thus, $D^J_\uo$ and $D^J_\ui$ have bounded right inverses and so do $D^{J'}_\uo$ and $D^{J'}_\ui$. From these inverses we construct in section \ref{sinv} and under the assumption  
 that $r<r_5(s_p,s_H,c_4)$, an inverse to $D^{J'}_{\ur}$. The dependence  
  of $c_4$ on $J'$ will not be fatal since a choice of a smaller $r$  
  does not increase the difference between $D^J_\ur$ and $D^{J'}_{\ur}$. The bounded inverse of the second gives a bounded inverse for the first.
\par Implicit function theorem can now be used by choosing $\xi_0=0$ and $u=u^r$. The result is the $J$-holomorphic curve $\exp_u \xi$, where
\[
\nwup{\xi} < c r^{1+\eps}.
\]
In particular, thanks to Sobolev embedding, the sup norm of $\xi$ is bounded, and consequently the difference between the holomorphic map obtained by perturbation of $\ur$ and $\ur$ itself will be of the order of $r^{1+\eps}$.
\begin{remark}\label{etrang}
Let $\eps \in ]0,1/3[$, let $0< \rho_1 < \rho_2<r_0(\eps)$, let $h^{\rho_1}$ and $h^{\rho_2}$ be curves obtained by theorem \ref{add}, \ie by applying proposition \ref{til} to the maps $u^{\rho_1}$ and $u^{\rho_2}$. Then, for a $K \in \rr$, if $\rho_1<\rho_2(1- K \rho_2^\eps)$ these two curves are at a positive Hausdorff distance. This is seen by looking at the strangling the approximated solutions have close to the gluing point. On one hand, theorem \ref{add} states that the Hausdorff distance from $h^{\rho_i}$ to $u^{\rho_i}$ is bounded by $O(\rho_i^{1+\eps})$. On the other hand, the distance from $u^{\rho_1}$ to $u^{\rho_2}$ is at least $K' (\rho_2 - \rho_1)$.
\par Furthermore, the implicit function theorem (as in \cite[Theorem A.3.3]{mds1}) indicates that the dependence of $h^\rho$ on $u^\rho$ is continuous. Thus, there exists a $\rho_0$ such that the $h^\rho$ for $\rho <\rho_0$ realize all possible strangling. The strangling for a given $\rho$ is however not precisely known. It would be tempting to make the gluing construction so that the resulting map has a fixed strangling, however the author could not find a well-behaved measure of strangling for a curve of class $W^{1,p}$. 
\end{remark}

\section{Chains of curves} \label{recinf}

This section is concerned with gluing (under certain assumptions) an infinite number of $J$-holomorphic curves in order to obtain a $J$-holomorphic cylinder. Although the method applies to general situations, we could content ourselves with the following setting. Assume three $J$-holomorphic curves intersect at three points, then there is a $J$-holomorphic cylinder that curls up around those curves. The main point of this section is to introduce a new norm, $\ell^\infty(L^p)$, on the base space. This will enable to treat the infinite number of gluing as if only two were happening. As such, it is a preparatory step for the interpolation construction. 

\subsection{Cylinder and $\ell^\infty(L^p)$ norms}\label{senreci}

The way the infinite number of gluing will be made is of course important. Let $\SR_i = \cp^1$ be compact Riemann surfaces, let $z_{i;0}$ and $z_{i;\infty} \in \SR_i$ be two marked points on each surface, and let $u^i: \SR_i \to M$ be $J$-holomorphic maps (for $i \in \zz$) such that $\forall i \in \zz, u^i(z_{i;0}) = u^{i-1}(z_{i-1;\infty})$. Finally, let $\CR = \rr \times S^1 = \cc / \zz$ be the $J$-holomorphic cylinder. This section will construct a $J$-holomorphic map $u^{(r_i)}: \CR \to M$ which is arbitrarily close to the $u^i$, that is $u^{(r_i)}$ restricted to $[i,i+1] \times S^1$ is close to $u^i$ when $\sup \{r_i\} \to 0$.
\par The space $\CR = \rr \times S^1 = \cc / \zz$ will be given a peculiar metric so that each segment $[i,i+1]$ resembles a sphere with two discs removed (see if necessary figure in the version available on the author's website). Let $(r_i) \in \liz{\rr_{>0}}$. Let $g_{(r_i)}$ be a family of metrics defined as follows. Let $i \in \zz$, then $g_{(r_i)}$ is the metric induced by the map $\mu_{i;r_i,r_{i+1}}$ which embeds $[i,i+1]\times S^1$ into the compact Riemann surface $\SR_i$ with the two discs $B_{r_i}(z_{i;0})$ and $B_{r_{i+1}}(z_{i+1;\infty})$ removed.
\par The volume of such a surface is infinite. Thus, $L^p$ norms are not expected to behave nicely. However a slight alteration will do. Let us consider the $\sup$ of the $L^p$ norms on annuli around each circle $\{i\}\times S^1$. Let $V$ be a vector bundle over $\CR$ (with a connection and a norm) and let $\xi: \CR \to V$ be a section, define 
\begin{equation} \label{lilp}
\begin{array}{rl}
\nr{\xi}_{\ell^\infty(L^p)}  &= \supp{n\in \zz} \nr{\xi}_{L^p([n-\frac{2}{3},n+\frac{2}{3}] \times S^1)},\\
\nr{\xi}_{\ell^\infty(\wup)} &= \supp{n\in \zz} \nr{\xi}_{\wup([n-\frac{2}{3},n+\frac{2}{3}] \times S^1)}.
\end{array}
\end{equation}
These norms will retain all the properties we need and will allow to look at the problem only one gluing at a time. It has been pointed to the author that similar norms are used \cite{tsu2}. A proof identical to that of lemma \ref{csob}, allows us to deduce that Sobolev embedding holds with a constant which does not depend on the parameters $r_{i}$ (given they are sufficiently small).  
\begin{lemma} \label{csobli}
Suppose $d_{\SR_i}(z_{i;0}, z_{i;\infty}) \geq c_0$. Given that $\supp{i \in \zz} r_i < \frac{c_0}{3}$, there exists a constant $s'_p \in \rr_{>0}$ such that
\[
s'_p:= \sup_{0 \neq f \in C^\infty(\CR)} \fr{\nr{f}_{L^\infty}}{\nr{f}_{\ell^\infty(\wup)}}.
\]
\end{lemma}
\begin{proof}
Each function can be decomposed as a sequence of functions on $\SR_i = \cp^1$ with $B_{r_n}(z_{n;0})$ and $B_{r_{n+1}}(z_{n+1;\infty})$ removed. The estimates follow from the fact that a ball with the Fubini-Study metric and one (or a fixed finite number of discs) removed has a Sobolev constant that remains bounded as the radius of the discs tends to $0$. See \cite{mds1} for details.
\end{proof}
The main result that will allow us to conclude is an adaptation of proposition \ref{til} to these norms.
\begin{e-proposition}\label{til2}
Let $\CR$ be one of the 1-dimensional non-compact complex manifolds described above. Let $p>2$. $\forall c_0, \exists \delta >0 $ such that for all volume forms $\dd \textrm{vol}_\CR$ on $\CR$ induced as above by the maps $\mu_{i;r_i,r_{i+1}}$, all continuous map $u$ such that $\dd u \in \ell^\infty(L^p)(\tg \CR,u^* \tg M)$, all $\xi_0 \in \ell^\infty(\wup)(\CR,u^*\tg M)$, and all $Q_u: \ell^\infty(L^p)(\CR, \Lambda^{0,1} \otimes_J u^*\tg M) \to \ell^\infty(\wup)(\CR,u^*\tg M)$ satisfying
\[
  \begin{array}{ccc}
    s'_p(\dd \textrm{vol}_\CR) \leq c_0, & \nr{\dd u}_{\ell^\infty(L^p)} \leq c_0, & \nr{\xi_0}_{\ell^\infty(\wup)} \leq \frac{\delta}{8}, \\
 \nr{ \db_J(\exp_u(\xi_0))}_{\ell^\infty(L^p)} \leq \frac{\delta}{4 c_0}, & D_u Q_u = \un, & \nr{Q_u} \leq c_0,
  \end{array}
\]
there exists a unique $\xi$ such that
\[
  \begin{array}[c]{ccc}
    \db_J(\exp_u(\xi_0+\xi))=0, & \nr{\xi+\xi_0}_{\ell^\infty(\wup)} \leq \delta, & \nr{\xi}_{\ell^\infty(\wup)} \leq 2 c_0 \pnr{\db_J(\exp_u(\xi_0))}_{\ell^\infty(L^p)}.
  \end{array}
\]
\end{e-proposition}
The proof will require the implicit function theorem in Banach spaces, as does the proof of proposition \ref{til}; it will be invoked again in section \ref{esphol}. 
\par The techniques are essentially the same as before, it suffices to insure that the curves $u^i$ and the points $m_i=u^i(z_{i;\infty})=u^{i+1}(z_{i+1;0})$ belong to a compact family, in other words that the parameters (\eg the radius $r_0$ below which the constructions can be performed, the kernel of the operators $D_{u^i}$, ...) remain controlled. For example, if the $u^i$ are just a finite number of curves infinitely repeated, the family is compact. Surjectivity in the sense of \cite[Theorem 10.1.2.iii]{mds1} is however much harder to get. Indeed, the kernel of the operators are (probably) no longer finite dimensional.  
\par Two types of constructions are possible at the points of intersection $m_i$, one can glue either as in \cite{mds1} or as in section \ref{chadd} of the present text. They are quite similar, except that the second uses stronger assumptions (yielding a more precise result). Indeed, in order to get the second construction, a condition on the tangent plane of the curves at the point of intersection is required.
\begin{e-definition}
Let $I\subset \zz$, then the curves $u^i$ are $I$-uniformly not tangent (at their points of intersection $m_i$) if $\forall i \in I$, there exist local charts $\psi_i:M \to \cc^m$ such that $\psi(m_i) = 0 \in \cc^n$ and $(\psi_i^* J) (0)= J_0$,
\[
\psi_i \circ u^i[1:z] = a_{i;\infty} z + O(\abs{z}^2) \quad \textrm{ and } \quad \psi_i \circ u^{i+1}[z:1] = a_{i+1;0}z + O(\abs{z}^2)
\] 
and $\exists d$ such that
\[
0<d < \underset{i \equiv i_0 \textrm{ mod k}}{\underset{\lambda \in \cc}{\inf}} \min (\abs{\lambda a_{i;0} + a_{i+1;\infty}}, \abs{a_{i;0} + \lambda a_{i+1;\infty}}).
\]
If this condition holds at every intersection (\ie $I=\zz$), we will simply say that they are uniformly not tangent.
\end{e-definition}
The following results is in two parts, depending on the transversality assumption made. The weaker, more standard transversality assumption is defined in the next section. 
\begin{theorem}\label{threcinf}
Let $J$ be an almost complex structure, and suppose it is regular in the sense of definition \ref{regtransi}. Let $z_{i;0} = [0:1] $ and $z_{i;\infty} = [1:0]\in \SR_i = \cp^1$. Let $u^i: \SR_i \to M$ be a compact family of $J$ holomorphic maps such that $u^i([0:1]) = u^{i+1}([1:0])$. Then there exist $c_2$ and $r_2 \in \rr_{>0}$ such that for all sequences $(r_i)_{i \in \zz}$ satisfying $r_{\sup{}} = \supp{i \in \zz} r_i \leq r_2$, there exists a $J$-holomorphic map $u^{(r_i)}$ such that the distance of $u^{(r_i)}(\CR)$ to $\cup u^i(\SR_i)$ is less than $c_2 r_2$. More precisely, 
\[
\forall z \in \mu_{i;r_i,r_{i+1}}^{-1}(\SR_i), \quad d_M(u^{(r_i)}(z), u^i(\mu_{i;r_i,r_{i+1}}(z))) \leq c_2 r_{\sup{}}.
\]
\indent If the curves are $I$-uniformly transverse, then there exists $r_3 \in \rr_{>0}$ such that for all sequence $(r_i)_{i \in \zz}$ satisfying $r_{\sup{}} = \supp{i \in \zz} r_i \leq r_3$ there exists a $J$-holomorphic map $v^{(r_i)}$ such that $\forall i \in I,$
\[
\begin{array}{rll}
  &\forall z \in (\phi_1 \circ \mu_{i;r_i,r_{i+1}})^{-1}\{A_{r_{i+1},r_{i+1}^{2/3}}(\infty) \},
    & \psi_i \circ v^{(r_i)} \circ \phi_1 \circ \mu_{i;r_i,r_{i+1}} (z) = a_{i;\infty} z + a_{i+1;0} \frac{r_{i+1}^2}{z}+ O(r_{\sup{}}^{1+\eps}) \\
\textrm{and}
  &\forall z \in \mu_{i+1;r_{i+1},r_{i+2}}^{-1} \{A_{r_{i+1},r_{i+1}^{2/3}}(0) \},
    & \psi_i \circ v^{(r_i)} \circ \mu_{i+1;r_{i+1},r_{i+2}} (z) = a_{i+1;0} z + a_{i;\infty} \frac{r_{i+1}^2}{z}+ O(r_{\sup{}}^{1+\eps})
\end{array}
\]
where $A_{r_1,r_2}(z_0) = B_{r_2}(z_0) \setminus B_{r_1}(z_0)$, $r_{\sup{}} = \supp{i \in \zz} r_i$ and $\psi_i:M \to \cc^m$ is a local chart that maps $m_i$ to $0$ and such that $(\psi_i^* J) (0)= J_0$.
\end{theorem} 
\par Let us begin by the proof of proposition \ref{til2}. It relies on \cite[Theorem A.3.3]{mds1}. It must be checked that the linearization of $\db_J$ does not vary too much depending on the point at which it is taken. Viewing this as Newton's method, it is the same as requiring the second derivative to be bounded. The methods used here are thus essentially the same as in \cite[{\S}3.5]{mds1}. Amongst other things, this part of the argument works even if $\SR_i \neq \cp^1$.
\par We introduce notations again: 
\[
\jo{X}_u = \ell^\infty(\wup)(\CR, u^*\tg M) \qquad \textrm{and} \qquad \jo{Y}_u = \ell^\infty(L^p)(\CR, \Lambda^{0,1}\otimes_J u^*\tg M).
\]
We are interested in the map $\jo{F}_u : \jo{X}_u \to \jo{Y}_u$ which is given by pulling back by parallel transport to $u$ the $1$-form $\db_J (\exp_\xi u)$. 
\begin{lemma}\label{bds} (\cf \cite[Proposition 3.5.3]{mds1})
Let $\CR$ be a manifold describe above, and let $p>2$. Then for all constants $c_0 >0$ there is a real number $c_1>0$ such that $\forall u \in \ell^\infty(\wup)(\CR,M), \forall \xi \in \ell^\infty(\wup)(\CR,u^*\tg M)$ and for all metrics satisfying
\[
\nr{\dd u }_{\ell^\infty(L^p)} \leq c_0, \qquad \nr{\xi}_{\ell^\infty} \leq c_0 \qquad \textrm{and} \qquad s'_p(\dd vol_\CR) \leq c_0,
\]
then
\[
\nr{\dd \jo{F}_u(\xi) - D_u} \leq c_1 \nr{\xi}_{\ell^\infty(\wup)}.
\]
where the norm on the left-hand side is the norm of linear operators $\jo{L}(\jo{X}_u,\jo{Y}_u)$.
\end{lemma}
The proof of this lemma is identical to that of \cite{mds1}, up to the change of $L^p$ norms for $\ell^\infty(L^p)$ norms. 
\begin{proof}[Proof of proposition \ref{til2}:] (\cf \cite[p.69]{mds1}) By assumption, $D_u$ has a bounded right inverse $Q_u$ ($\nr{Q_u} \leq c_0$). Let $c_1$ be the constant from lemma \ref{bds}, and let $\delta \in ]0,1[$ be such that $c_1 \delta < 1 /2c_0$. Then lemma \ref{bds} insures that $\nr{\dd \jo{F}_u(\xi) - D_u} \leq 1/2c_0$ if $\nr{\xi} \leq \delta$. The assumptions of the implicit function theorem \cite[Proposition A.3.4]{mds1} are consequently satisfied (with $X = \jo{X}_u$, $Y = \jo{Y}_u$, $f = \jo{F}_u$, $x_0=0$, $c_0 = c$ and the same $\delta$).
\end{proof}

\subsection{Transversality and right inverse}

Before the theorem can be put to good use, it is better to check that the surjectivity of linearized operators holds in a reasonable class of spaces. Recall that $\jo{M}^*(A^i,\SR_i; J)$ is the space of $J$-holomorphic maps $\SR_i \to M$ that are somewhere injective and represent the homology class $A^i$ in $H^2(M)$.
\begin{e-definition}\label{regtransi}
Let $\forall i \in \zz, A^i \in H^2(M,\zz)$, let $\SR_i$ be Riemann surfaces. The structure $J$ will be said regular for $(A^i)_{i \in \zz}$ and $(\SR_i)_{i \in \zz}$ if $J \in \cap_{i \in \zz} J_{reg}(\SR_i,A^i)$ and if the evaluation map:
\[
\begin{array}{rclc}
\textit{ev}_\zz:& \liz{ \jo{M}^*(A,\CR;J)}   &\to    & \liz{M \times M}  \\
                &           (u^i)_{i \in \zz} &\mapsto& (u^i(z_{i;\infty}), u^{i+1}(z_{i+1;0}))_{i \in \zz}
 \end{array}
\]
is transverse to $\Delta^\zz = \liz{\Delta}$ where $\Delta = \{ (m,m) \subset M \times M \}$. The set of structures satisfying these conditions will be written $\jo{J}_{reg}((\SR_i)_{i \in \zz},(A^i)_{i \in \zz})$ or more simply $\jo{J}_{reg}(\SR_i,A^i)$.
\end{e-definition}
\indent Although from our point of view the almost complex structure is given, it is wise to show that structures that are regular are abundant. As the intersection of a countable number of dense open subsets is still a set of the second category, to show that $\jo{J}_{reg}(\SR_i,A^i)$ is of the second category only requires the study of $\dd \textit{ev}_\zz$.
\par This would require an adaptation of theorem \cite[Theorem 6.3.1]{mds1} (which insures transversality for curves glued according to a finite tree). $\zz$ can be seen as an infinite tree, and so the question can be asked in general for an infinite tree $T$ of bounded degree. It might be tempting to proceed as follows: take an increasing sequence of finite subtrees of $T$, say $\{T_i\}$. Thanks to theorem \cite[Theorem 6.3.1]{mds1} the set of structures for which the evaluation on the tree $T_i$ is transversal is of the second category. The intersection of these sets should yield a set of the second category. 
\par In what follows we shall suppose that the structure $J$ on $M$ is regular in the sense of definition \ref{regtransi}. This assumption is not so strong, especially since, in the cases of interest, the $u^i$ will be a periodic sequence of curves (\ie $\exists n\in \zz_{>0}$ such that $u^i = u^j$ if $i \equiv j \modu (n)$). Thus the {\it a priori} infinite condition of definition \ref{regtransi} are actually finite. Let us assume that each curve $u^i \in \jo{M}^*(A^i, \SR_i;J)$ is such that $\dd \ev_0$ (its evaluation at $0 \in \SR_i = \cp^1$) is surjective. In other words, for each curve it is possible to  choose an infinitesimal perturbation (which is also $J$-holomorphic) in such a way that this perturbation displaces $u^i(0)$ in any chosen direction (note that we do not make any assumption on the effect of this perturbation at $\infty \in \cp^1$). Then the evaluation is surjective. Indeed, if we are given a infinitesimal displacement at each point of the gluing, making it equal to the difference between the displacement of $u^i(\infty)$ and of $u^{i+1}(0)$ amounts to solve $n$ equations knowing that in each equation we can fix the value of a term (the one coming from the displacement of $u^{i}(0)$). Since it is a finite system (by periodicity) it is solvable. Whence the surjectivity of evaluation.
\par Finally, note that in $\cp^n$ endowed with its usual structure, these assumptions hold (at least for some $A^i$) since between any two distinct points of $\cp^n$ there is a line (or a conic). 
\par Let's define the moduli space
\[
\jo{M}^*(A^i,\SR_i;J)= \{(u^i)_{i \in \zz} \in \liz{ \jo{M}^*(A,\CR;J)} | \forall i \in \zz, u^i(z_{i;\infty}) = u^{i+1}(z_{i+1;0}) \}.
\]
It is not excluded that the dimension of this space might be finite. For example, if almost all $A^i$ have a trivial first Chern class, it might happen that the dimension of the modular space is $2n+2 \sum c_1(A^i)$. In the present context, we will be interested in a subset of the moduli space when $\SR_i = \cp^1$:
\[
\jo{M}^*_\zz(C) := \jo{M}^*(A^i;J;C) := \{(u^i) \in \jo{M}^*(A^i,\SR_i = \cp^1;J)  | \pnr{\dd u^i}_{L^\infty} \leq C , \forall i \in \zz \}.
\]
\par What matters is that transversality of definition \ref{regtransi} implies surjectivity of the linearized operator even if it is restricted to vector fields who do not alter the intersection property. Recall that for $u^i:\SR_i \to M$, 
\[
\begin{array}{rl}
  \wup_{u^i}  &= \wup(\SR_i, u^{i*}\tg M) \\
L^p_{u^i} &= L^p(\SR_i, \Lambda^{0,1}\tg^*\SR_i \otimes_J u^{i*}\tg M).
\end{array}
\]
Given $u^i: \SR_i \to M$, such that $u^i(z_{i;\infty}) = u^{i+1}(z_{i+1;0})$, denote by
\[
\wup_{u^\zz} := \left\{ (\xi^i)_{i \in \zz} \in \timesi \wup_{u^i} | \xi^i(z_{i;\infty}) = \xi^{i+1}(z_{i+1;0})  \right\}.
\]
(The evaluation of $\wup$ sections makes sense since $p>2$.) 
\begin{lemma} \label{surji}
Suppose $J$ is regular in the sense of definition \ref{regtransi}, \ie the operators $D_{u^i}: \wup_{u^i} \to L^p_{u^i}$ are surjective and {\it ev}$_\zz$ is transverse, then the operator
  \[
  \begin{array}{rccc}
  D_{u^\zz} :& \wup_{u^\zz}        & \to     & \timesi  L^p_{u^i} \\
             & (\xi^i)_{i \in \zz} & \mapsto & (D_{u^i} \xi^i)_{i \in \zz}
  \end{array}
  \]
  is surjective. 
\end{lemma}
\begin{proof}
Let $\eta^i \in L^p_{u^i}$ (where $i \in \zz$). Each of the $D_{u^i}$ being surjective, there exist $\xi^i \in \wup_{u^i}$ such that $D_{u^i} \xi^i = \eta^i$. Since the evaluation is transverse to the diagonal, choose $\zeta^i \in \tg_{u^i} \jo{M}^*(A^i;J)$ so that
 \[
 \dd \textit{ev}_\zz ((u^i)_{i \in \zz})((\zeta^i)_{i \in \zz})= (\zeta^i(z_{i;\infty}),\zeta^{i+1}(z_{i+1;0}))_{i \in \zz} \in \timesi \bigg( (\xi^i(z_{i;\infty}),\xi^{i+1}(z_{i+1;0}))+ \tg_{(m_i,m_i)}\Delta \bigg)
 \]
where {\it ev}$_\zz$ is the map defined in \ref{regtransi}, $m_i=u^i(z_{i;\infty})=u^{i+1}(z_{i+1;0})$ and $\Delta \subset M\times M$ is the diagonal. Then $(\xi^i-\zeta^i)_{i \in \zz}$ is an element of $\wup_{u^\zz}$ whose image by $D_{u^\zz}$ is also $(\eta^i)_{i \in \zz}$.
\end{proof}
\indent In order to use proposition \ref{til2} we have to describe an approximate solution and show that it has bounded right inverse. To do so, two choices are possible: either the method of \cite[{\S}10]{mds1} or the one from section \ref{chadd} (if the curves $u^i$ are transversal at their point of intersection, so in particular $\dim_\rr M \geq 4$). The second is of interest since by remark \ref{etrang} it could allow to prescribe different characteristics  of curves (strangling), an idea that will be used again in section \ref{esphol}. These two situations are dealt with in an identical fashion. The main point is to notice that in the constructions the approximate solutions differ from the initial curves only in a neighborhood of the points of intersection. Similarly, the approximate inverses only differ from a true inverse in those neighborhoods.
\par Suppose that we are trying to use the construction of section \ref{chadd}, some problem might arise from the fact that a subsequence of $\{m_i\}$ may be arbitrarily close. Modifying the almost complex structure $J$ would then be a problem. To avoid this, it is again necessary to introduce structures which depend on a point of the domain. We will not detail this argument. Furthermore, the construction used a reparametrization of one of the curves; we present what this means in the present context. 
\par Suppose we are in the case where $z_{i;0} = [0:1] =0$ and $z_{i;\infty} = [1:0] = \infty \in \SR_i = \cp^1$. Let $\phi_r: \cp^1 \to \cp^1$ be defined by $\phi_r(z) = r^2/z$. Then, the condition of intersection is $u^i \circ \phi_1 (0) = u^{i+1}(0)=m_i$, and the local expansion in a chart $\psi_i:M \to \cc^m$ which sends $m_i$ to $0$ and such that $(\psi_i^* J) (0)= i$ is
\[
\psi_i \circ u^{i+1}[z:1] = a_{i+1;0}z + O(\abs{z}^2) \quad \textrm{ and } \quad \psi_i \circ u^i[1:z] = \psi_i \circ u^i\circ \phi_1 [z:1] = a_{i;\infty} z + O(\abs{z}^2).
\]
The ring $A_{r_{i+1}^{4/3},r_{i+1}^{2/3}}$ corresponds on $\CR$ to the $z \in [i,i+1] \times S^1$ such that $\phi_1 \circ \mu_{i;r_i,r_{i+1}}(z) < r_{i+1}^{2/3}$ and to the $z \in [i+1,i+2] \times S^1$ such that $\mu_{i+1;r_{i+1},r_{i+2}}(z) < r_{i+1}^{2/3}$.
\par The arguments used for the $L^p$ norm will be adapted without pain to the $\ell^\infty(L^p)$ context: in this norm there is at most one gluing to consider at a time. It suffices to check that the curves $u^i$ and the points $m_i=u^i(z_{i;\infty})=u^{i+1}(z_{i+1;0})$ belong to a compact family (\eg a finite family). 
\par We now transpose the methods of {\S}\ref{sinv} to conclude. 
\begin{proof}[Proof of theorem \ref{threcinf}: ]
With a small deformation the $J$-holomorphic curves $u^i$ can be modified into maps $u^{i;r_i,r_{i+1}}$; this deformation is identical to the one which changes $u^0$ into $u^{0,r}$ except it takes places at two points of $\cp^1$. Thus, the operators $D_{u^i}$, and their inverses, are close to $D_{u^{i;r_i,r_{i+1}}}$. The opertor $D_{u^\zz}$ also is also close (in the norm of linear maps $\ell^\infty(\wup) \to \ell^\infty(L^p)$) to an operator $D_{u^{\zz;(r_i)}}$ where the $D_{u^{i;r_i,r_{i+1}}}$ take place of the $D_{u^i}$. It is also surjective and their inverses are close.
\par We describe the map $u^{(r_i)}:\CR \to M$ (we will write $u= u^{(r_i)}$ for short in this paragraph) which will be an approximate solution, in the sense that $\pnr{\db_J u}_{\ell^\infty(L^p)}$ is small, and that $D_u$ will have a bounded right inverse. It will be defined by composing the maps $\mu_{i;r_i,r_{i+1}}:[i,i+1]\times S^1 \to \SR_i$ with the maps $u^{i;r_i,r_{i+1}}:\SR_i \to M$. Then a $(0,1)$-form, say $\eta$, along $u$ can be cut in pieces to give rise to $\eta^i$ along each $u^{i;r_i,r_{i+1}}$ (by extending by $0$, that is in an analogous way as \eqref{etar} where $\eta^0$ and $\eta^1$ were obtained from $\eta$). From these $\eta^i$, the inverse of $D_{u^{\zz;(r_i)}}$ will give vector fields along the $u^{i;r_i,r_{i+1}}$, say $\xi^i$. These vector fields can be glued by a surgery (which copies the definition of $\xi^r$ in \eqref{xir}) to get a vector field $\xi^{(r_i)}$ along $u$. The computation made in \eqref{brnpinv} still works out in an identical fashion at each point of intersection. By definition of the $\ell^\infty(L^p)$ norm, this construction produces an approximate inverse to $D_u$. By the technique used in \eqref{qur}, a true bounded inverse is then found. This allows the use of proposition \ref{til2} and finishes the proof.
\end{proof}

\section{Interpolation and its consequences} \label{esphol}

In this section, we give an example of a space of pseudo-holomorphic maps which is of positive mean dimension. As before the cylinder will be noted $\CR = \rr \times S^1 = \cc / \zz$. We will assume that the of curves $u^i$ is of finite type (periodic), in the sense that only a finite number of distinct maps are described as $i$ runs over $\zz$. The theorem \ref{receval} could also be proven using the gluing introduced in \cite{mds1}. However, in order to prove proposition \ref{imcjh} and to apply lemma \ref{upprec}, it is necessary to have approximate solutions which are injective (with a discrete set of exceptions). This is incompatible with an approximate solution which is constant on a whole ring.

\subsection{The interpolation theorem}\label{hypint}

Apart from the interpolation itself, the results of theorem \ref{inter} have now been covered. As we cannot unfortunately gain information from the parameters $r_i$, we shall throughout this section take them to be all equal $r_i=:r$. We will also use the notation $ \mu_{i;r} := \mu_{i;r_i,r_{i+1}}$.
\par Let us recall all the assumptions we shall need. 
\begin{enumerate}\renewcommand{\labelenumi}{{\normalfont H\arabic{enumi} -}}
\item For $i \in \zz$, there exists $N \in \zz_{>2}$ and $J$-holomorphic curves $u^i: \cp^1 \to M$ with $p_i:=u^i(\infty) = u^{i+1}(0)$ and $u^{i+N} = u^i$.
\item The curves at $p_i$ are not tangent. (Necessary for the results of \S{}\ref{ntfam} and \S{}\ref{p:simp})
\item $J$ is supposed regular in the sense of definition \ref{regtransi} and of class $C^2$.
\item One of the maps, say $u^j$, will be assumed to live in a family that covers the neighborhood of some point $u^j(z_*)$ for $z_* \in \cp^1$ (it is a ``free curve'').
\end{enumerate}
\par This last assumption is more precisely stated as follows. For a fixed $j \in \{1,2,\ldots N\}$, there exists a point $z_* \in \cp^1$ and a family of $\wup$ vector fields along $u^j$ which belong to the kernel of $D_{u^j}$ and such that the map defined by $\xi \mapsto \exp_{u^j(z_*)}{\xi(z_*)}$ is surjective on a neighborhood of $m_* = u^j(z_*)$. Thus, to $x \in \tg_{m_*} M$ we will associate the vector field $X_x \in \ker D_{u^j} \subset \wup(\cp^1,(u^j)^*\tg M)$ such that $X_x(z_*)=x$. In other words, we need to make the assumption that the differential of the map given by evaluation at $z_*$ is surjective on the kernel of $D_{u^j}$.
\par Remark that this assumption is close to the transversality of the evaluation map in definition \ref{regtransi}. Let $\ev_{j+N\zz}$ be the evaluation at $(u^i)$ in $z_*$ when $i \equiv j \modu N$. To ask that
\[
D_{u^\zz} \oplus \dd \ev_{j+N\zz} : \timesi \wup_{u^i}  \to (\timesi L^p_{u^i}) \oplus (\timesi \tg_{u^{j+Ni}(z_*)} M )
\]
is surjective is, given that $J$ is regular in the sense of definition \ref{regtransi}, equivalent to ask that the restriction of $\dd \ev_{u^{j+Ni}}$ to the subspace $\ker D_{u^{j+Ni}}$ be surjective. This condition is naturally expressed in the vocabulary of transversality. Indeed, if in the construction of section \ref{recinf}, it is required, in addition to the gluing between the curves in the chain, to glue another curve at a point $z_*$ (\eg a constant curve), then regularity for this gluing scheme (which obtained from $\zz$ as a tree, by adding a leaf to the integers $j+N\zz$) implies the surjectivity of $\dd \ev_{u^{j+Ni}} \big|_{\ker D_{u^{j+Ni}}}$. This way of presenting our assumption indicates that it is not significantly stronger the one made in the preceding section, particularly for a finite family of curves. For example, it holds for a finite number of curves with appropriate intersection in $\cp^n$ with its usual complex structure.
\par Finally, in order to remain in the setting of a compact family of maps $\{u^i\}$, it will be necessary to restrict to a sufficiently small ball $B_{m_*} \subset \tg_{m_*} M$ such that for all $x \in B_{m_*}$ the curves $\exp_{u^j}X_x$ form a compact family.
\begin{theorem} \label{receval}
Let $(M,J)$ be an almost-complex manifold. Let $u^1, \ldots, u^N$ be a finite family of $J$-holomorphic curves $u^i:\cp^1 \to M$ such that $u^i(\infty) = u^j(0)$ when $j \equiv i+1 \modu N$. Suppose that $J$ is regular in the sense of definition \ref{regtransi} and that $u^j$ is deformable. Let $z_* \in \cp^1 \setminus\{0,\infty\}$ be a marked point and let $m_* = u^j(z_*) \in M$ be its image. Suppose that the curves are uniformly transverse. There exists $c$ and $R \in \rr_>0$ such that for any sequence $\{r_i\}$ satisfying $r_{\sup} = \sup r_i \leq R$, the exists a neighborhood $V_{m_*}$ of $m_*$ such that for all sequence of points $\{m_k\}_{k \in \zz}$ in $V_{m_*}$ there exists a $J$-holomorphic cylinder $u: \CR \to M$ satisfying the following properties: \\
\vspace*{0.5ex}
\hspace*{2ex} R1 - $u$ passes by the prescribed points, \ie 
\[
u(z_{k,*})=m_k,
\]
where $z_{k,*} = \mu_{j+Nk;r}(z_*)$, \\
\vspace*{0.5ex}
\hspace*{2ex} R2 - the behavior of $u$ near the gluing points is as follows: 
\[
\begin{array}{rll}
  &\forall z \in (\phi_1 \circ \mu_{i;r})^{-1}\{A_{r_{i+1},r_{i+1}^{2/3}}(\infty) \},
    & \psi_i \circ u \circ \phi_1 \circ \mu_{i;r} (z) = a_{i;\infty} z + a_{i+1;0} \frac{r_{i+1}^2}{z}+ O(r_{\sup{}}^{1+\eps}) \\
\textrm{and}
  &\forall z \in \mu_{i+1;r}^{-1} \{A_{r_{i+1},r_{i+1}^{2/3}}(0) \},
    & \psi_i \circ u \circ \mu_{i+1;r} (z) = a_{i+1;0} z + a_{i;\infty} \frac{r_{i+1}^2}{z}+ O(r_{\sup{}}^{1+\eps})
\end{array}
\]
where $A_{r_1,r_2}(z_0) = B_{r_2}(z_0) \setminus B_{r_1}(z_0)$, and $\psi_i:M \to \cc^m$ is a local chart that maps $m_i$ to $0$ and such that $(\psi_i^* J) (0)= J_0$.\\
\vspace*{0.5ex}
\hspace*{2ex} R3 - $u$ is close to the curves $u^i$ (or $\exp_{u^j}X$ if it is the deformable curve):
\[
\forall z \in \mu_{i;r}^{-1}(\Sigma_i), \quad d_M(u^{(r);0}(z), u^i(\mu_{i;r}(z))) \leq c ( r_{\sup{}}^{1+\eps} + \delta_i)
\]
where $\delta_i = d_M(m_k,m_*)$ if $i = j+ kN$ and $\delta_i = 0$ else.
\end{theorem} 
(The maps $\mu_{i;r}$ are the same as the one introduced in {\S}\ref{senreci}.)

\subsection{Implicit function theorem again}
As before we are trying to find a solution to an equation containing a non linear term by an implicit function theorem. However, in addition to $\db_J u =0$, we have to satisfy a sequence of punctual constraints. As we shall see these will not have a significative impact on the arguments. In order to describe the situation, note $z_{k;*}$ the marked points on the cylinder $\CR$ (they will be chosen to be equal to $\mu_{j+Nk;r}(z_*)$ later on), and let $\ev_*:\jo{M}_\CR \to \liz{M}$ be the evaluation map at these points $z_{k;*}$. Even if $\ev_*$ takes value in $M$, we are in a situation where only the curve $u$ and its perturbation by a vector field will intervene. 
\par Since we need the vector fields to be of $\ell^\infty$ norm smaller than the injectivity radius so that the evaluation of $\exp_u \xi$ makes sense, it is better to see the target space of $\ev_*$ as a product of balls in the tangent plane. Let $\tg_{\zz;*} M = \timesi \tg_{u(z_{i;*})}M $, the elements $w \in \tg_{\zz;*} M $ will sometime be written as $w=(w_i)_{i \in \zz}$. 
\par This understood, $\ev_*$ defined in a neighborhood of $u$ takes values in $\tg_{\zz;*} M$. This is actually a linear map if we look at the neighborhood of $u$ as given by vector fields along $u$. The equations to solve are $\db_J u =0$ and $\ev_* u = w$ for some $w \in \tg_{\zz;*} M$. This said, it remains to use the implicit function theorem \cite[Proposition A.3.4]{mds1}.
\begin{e-proposition}\label{til3}
Let $\CR$ be the cylinder and let $p>2$. $\forall c_0, \exists \delta >0 $ such that for any volume form $\dd \textrm{vol}_\CR$ on $\CR$ induced by the $\mu_{i;r}$, any continuous map $u$ and such that $\dd u \in \ell^\infty(L^p)(\tg \CR,u^* \tg M)$, all $\xi_0 \in \ell^\infty(\wup)(\CR,u^*\tg M)$, and all $T_u: \ell^\infty(L^p)(\CR, \Lambda^{0,1} \otimes_J u^*\tg M) \oplus \tg_{\zz,*} M \to \ell^\infty(\wup)(\CR,u^*\tg M) \oplus \tg_{\zz;*} M$ satisfying
\[
  \begin{array}{c}
    s'_p(\dd \textrm{vol}_\CR) \leq c_0, \quad \nr{\dd u}_{\ell^\infty(L^p)} \leq c_0, \quad \nr{\xi_0}_{\ell^\infty(\wup)} \leq \frac{\delta}{8}, \\
 \begin{array}{cc}
  D_u T_u = \un, & \nr{T_u} \leq c_0,\\
 \pnr{ \db_J(\exp_u(\xi_0))}_{\ell^\infty(L^p)} \leq \frac{\delta}{8 c_0}, &\nr{ \xi_0(z_{k;*}) - w_k }_{\ell^\infty} \leq \frac{\delta}{8 c_0}
  \end{array}
\end{array}
\]
there exists a unique $\xi$ such that 
\[
  \begin{array}{c}
    \db_J(\exp_u(\xi_0+\xi))=0, \qquad \xi_0(z_{k;*}) + \xi(z_{k;*}) = w_k, \qquad \nr{\xi+\xi_0}_{\ell^\infty(\wup)} \leq \delta, \\ 
    \nr{\xi}_{\ell^\infty(\wup)} \leq 2 c_0 \big( \pnr{\db_J(\exp_u(\xi_0))}_{\ell^\infty(L^p)} +  \nr{ \xi_0(z_{k;*}) - w_k }_{\ell^\infty} \big).
  \end{array}
\]
\end{e-proposition}
\begin{proof}
We proceed in the same fashion as in the proof of theorem \ref{til2}. Let $c_1$ be the constant of lemma \ref{bds}, and let $\delta \in ]0,1[$ be such that $c_1 \delta < 1 /2c_0$. Then lemma \ref{bds} insures that $\nr{\dd \jo{F}_u(\xi) - D_u} \leq 1/2c_0$ when $\nr{\xi} \leq \delta$. 
The map $\ev_* : U \to (\tg_{m_*} M)^\zz$ is defined for $U \subset \jo{X}_u$ the open set of vector fields whose $\ell^\infty$ norm is less than the injectivity radius. 
With these notations, $\dd \ev_* (\xi) = \dd \ev_* (0)$, and no estimate on the second derivative is required. 
\par The implicit function theorem of \cite[Proposition A.3.4]{mds1} will be used with the following notations: $w$ is an element of $(\tg_{m_*} M)^\zz$ contained in the image of $\ev_*$, 
\[
X = \jo{X}_u, \quad Y = \jo{Y}_u \oplus (\tg_{m_*} M)^\zz, \quad f = (\jo{F}_u, \ev_* - w), \quad x_0=0, \quad c_0 = c 
\]
and with $\delta$ the minimum of the $\delta$ above and of the real number $\delta'$ such that $\nr{\xi}_{L^\infty}$ is less than the injectivity radius of de $M$. Note that $\dd f_x - \dd f_{x_0}$ is bounded by lemma \ref{bds} and as $\dd \ev_* (\xi) = \dd \ev_* (0)$.
\end{proof}
\par In order to prove theorem \ref{receval}, the approximate solution to $f=0$ must be made and the operator $D_u \oplus \ev_*$ must be shown to have a bounded right inverse. Let us go back to $D_{u^\zz}\oplus \ev_{j+N\zz}$. In section \ref{recinf}, it was important that, under the assumption of the regularity of $J$, the map $D_{u^\zz}$ possesses a bounded right inverse (for the $\ell^\infty(L^p)$ and $\ell^\infty(\wup)$ norms). Call this inverse $Q_{u^\zz}$. Let us show that this allows us to construct an inverse to $D_{u^\zz} \oplus \dd \ev_{j+N\zz}$.
\begin{lemma}
If $D_{u^\zz}$ has a bounded right inverse and H4 holds, then a bounded inverse to $D_{u^\zz} \oplus \dd \ev_{j+N\zz}$ exists.
\end{lemma}
\begin{proof}
The assumption was designed so that, for the structure $J$ given, $\dd \ev_{u^j} \big|_{\ker D_{u^{j}}}$ is surjective on $\tg_{u^j(z_*)} M$. Thus there exists a map $q_j:\tg_{u^j(z_*)} M \to \ker D_{u^{j}}$ such that $\dd \ev_{u^j} \circ q_j = \Id$. This map is bounded since its domain is finite dimensional. Let us introduce $(\tg_{m_*} M)^\zz = \timesi \tg_{u^{j+Ni}(z_*)} M$. Recall that $u^{j+Ni}=u^j$ and $u^j(z_*)=m_*$. Thus, the map
\[
q: (\tg_{m_*} M)^\zz \to \ker D_{u^{j+Ni}}
\]
which reproduce $q_j$ on each factor is bounded from $\ell^\infty(|\cdot|) \to \ell^\infty(\wup)$ where $|\cdot|$ denotes a norm on $\tg_{u^j(z_*)} M$ and $\ell^\infty(|\cdot|)$ is the supremum of these norm on the product. Let $\eta \in \wup_{u^\zz}$ and $w \in (\tg_{m_*} M)^\zz$, define $T: L^p_{u^\zz} \oplus (\tg_{m_*} M)^\zz \to \wup_{u^\zz}$ by
\[
T(\eta, w) = Q_{u^\zz} \eta + q (w - \dd \ev_{u^{j+N\zz}} Q_{u^\zz} \eta).
\]
Since $D_{u^\zz} q =0$, $D_{u^\zz} T(\eta, w) = \eta$ and $\dd \ev_{u^{j+N\zz}} T(\eta, w) = w$. $T$ is the required right inverse to $D_{u^\zz} \oplus \dd \ev_{u^{j+N\zz}}$.
\end{proof}
\indent The following proof is a small modification of the proof of theorem \ref{threcinf}. 
\begin{proof}[Proof of theorem \ref{receval}: ]
We start by describing the approximate solution $u^{r;(w_k)}:\CR \to M$. The points $z_{*,k}$ will be chosen {\it a posteriori} as they will depend on the parameter $r$. This dependence could certainly be avoided by perturbing again the approximate solution, but this would require unnecessary estimates. As described at the end of section \ref{recinf}, $u^{r;(w_k)}$ will be defined by composing the maps $\mu_{i;r}:[i,i+1]\times S^1 \to \Sigma_i$ with $u^{i;r,r}:\Sigma_i \to M$ if $i \nequiv j \modu N$. When $i = j +Nk$, we will first deform $u^j$ by the vector field $X_{w_k}$ in a map $\exp_{u^j} X_{w_k}$, before it is deformed into the a map $u^{j;r,r}:\Sigma_j \to M$ (those are the small $\wup$ deformations defined similarly to $\uor$ from $\uo$ in section \ref{sinv}). 
\par The $(0,1)$-form $\eta$ along $u$ will be split into $\eta^i$ along the $u^{i;r,r}$ by extending with $0$ where it is not defined. From these $\eta^i$, we obtain the vector fields $\xi^i$ along the $u^{i;r,r}$ thanks to the inverse of $D_{u^{\zz;(r_i)}} \oplus \dd \ev_{j+N\zz}$. The $\xi^i$ will have the property that $D_{u^{i;r,r}} \xi^i = \eta^i$ and that if $i = j+Nk$ then $\xi^i(z_*) = 0$. To obtain a vector field $\xi^{r;(w_k)}$ along $u$, it remains to glue these fields as in section \ref{sinv} and in the proof of theorem \ref{threcinf}. Indeed, if $r$ is sufficiently small so that the point $z_*$ will not be contained in the region where the gluing of the vector fields take place, this method gives a $\xi^{r;(w_k)}$ which is $0$ at certain points, which means that the curve displaced by this vector field will take the prescribed value. 
\par The points $z_{k;*}$ are determined only at this step. First, the $r$ is chosen so as to satisfy the constraints mentioned so far but also so that theorem \ref{threcinf} applies. Then, we fix 
\[
z_{k;*} = \mu_{j+Nk;r}^{-1}(z_*). 
\]
Thus, $\xi^{r;(w_k)}(z_{k;*})=0$. Furthermore, the computation of \eqref{brnpinv} remains identical. This construction is an exact inverse for $\ev_*$. It also an approximate inverse to $D_u$ (for the $\ell^\infty(L^p)$ norm) since $r$ has been so chosen. Consequently, this is an approximate inverse for $D_u \oplus \ev_*$. The proper inverse is then obtained and proposition \ref{til3} applies to yield theorem \ref{receval}.
\end{proof}

\subsection{Non-triviality} \label{S:nt}

\par Theorem \ref{receval} shows the existence of a family of pseudo-holomorphic maps which can be parametrized as follows. Note by $\jo{R}: \rr_{>0} \times  B_{m_*}^\zz \to \jo{M}_\CR$ the map obtained by theorem \ref{receval}. In a neighborhood $V_{m_*}$ of a point $m_* = u^j(z_*)$ the map $\jo{R}_r\big((w_k)\big):\CR \to M$ is characterized by the value it takes at $z_{k;*}$. Note that since we took all the parameters $r_i$ to be equal, $z_{k;*}$ is the point $z_{0;*}$ translated by $iNk$. This understood, for fixed $r$, the maps obtained by theorem \ref{receval} are characterized by $m_k = \exp_{m_*} w_k \in V_{m_*}$.
\par Recall that in this only in \S{}\ref{ntfam} and \S{}\ref{p:simp} that the linear independence of the tangent of the curves at intersection points will be used.

\subsubsection{Distinct images.} \label{ntfam}
Another interesting property of this family of $J$-holomorphic applications resides in the fact that, under appropriate assumptions, two curves obtained from $\jo{R}$ have the same image only if they differ by an automorphism. 
\begin{e-proposition}\label{imcjh}
Let $u^i$ where $i = 1, \dots, N$ such that there exists $r_0$ and $B_{m_*} \subset \tg_{m_*} M$ such that $\forall r < r_0$ and $\forall w \in \ell^\infty(B_{m_*})$ the number of points where $u^{r;w}$ is not injective is finite. Then there exists $r_1 < r_0$ and $C \in \rr_{>0}$ such that for all $r<r_1$ and all $w_1,w_2 \in \ell^\infty(B_{m_*})$ such that $\nr{u^{r;w_1}-u^{r;w_2}}_{C^0} < Cr_1$, if $u_1 =\jo{R}_r(w_1)$ and $u_2 =\jo{R}_r(w_2)$ possess the same image then they differ by the precomposition of an automorphism. 
\end{e-proposition}
\begin{proof}
Introduce
\[
\Gamma = \{(z_1,z_2)\subset \CR \times \CR | u_1(z_1) = u_2(z_2) \}.
\]
This is an analytic set (\cf \cite[Proposition 5 and its remark]{Sik}), which is moreover complex, and, since the two curves have same image, of (complex) dimension 1. Note $\rho(r) = O(r^{1+\eps})$ the maximum of the $C^0$ distances between $u^{r;w_k}$ and $u_k =\jo{R}_r(w_k)$. Choose $r_1$ so that $B_{4\rho(r)} \big(u^{r;w_k}(z) \big) \cap u^{r;w_k}(\CR)$ be isomorphic to discs for all $r \leq r_1$. Next, take $C$ so that $C r_1 < 4 \rho(r_1) - 2 \rho(r)$ (for example $C= 2 \rho(r_1)/r_1$). Then, $\nr{u_1(z)-u_2(z)}_{C^0} \leq 2 \rho(r) + C r_1 < 4\rho(r_1)$.
\par Let $\Delta \subset \CR \times \CR$ be the diagonal and let $U_\rho \Delta$ a $\rho$-neighborhood of the latter. Then $\Gamma$ is close to $\Delta' = \cup_{z \in Z} \big( \Delta + z + i N \zz \big)$ where $Z$ is the set of points where the $u^{r;w_k}$ are not injective and $\Delta + c$ is a short notation for the diagonal translated along one of its factors in the product (\ie the set of pairs $(z+c,z)$). These choices made, $\Gamma$ is contained in a neighborhood of these translated diagonals, $U_{4 \rho(r_1)}\Delta'$. 
\par The map $s:(z_1,z_2) \mapsto (z_1, z_2-z_1)$ is an isomorphism of $\CR\times\CR$ on itself which sends the neighborhood $U_{\rho}\Delta$ on $\CR \times D_{\rho}$ (where $D_{\rho}$ is the disc of radius $\rho$). Let $\pi_k: \Gamma \to \CR$ be the projections on each factors. Given that the curves have the same image, these maps are surjective. Thus, $\pi_1 \circ s (\Gamma) \subset \CR$ and $\pi_2 \circ s (\Gamma) \subset D_\rho $. 
 Let $\Gamma_0$ be a connected component of $\Gamma$; this is a closed analytic complex set of dimension 1. So is $s(\Gamma_0)$ which is contained in $\CR \times D_{4 \rho(r_1)}$. This analytic set lifts to a subset of $\cc \times D_{4 \rho(r_1)}$. Describing this set by equations with holomorphic coefficients, one sees that the coefficient must be constant. Consequently this is a line. Thus, $\Gamma_0$ is contained in a translate of the diagonal; in other words for $z_1 \in \pi_1(\Gamma_0)$, $u_1(z_1) = u_2(z_1+c)$. As $\pi_1(\Gamma_0)$ is $\CR$ (note that by the uniqueness of the extension of $J$-holomorphic maps, it would suffice to have it non empty), this means that $u_1(z) = u_2(z+c)$.
\end{proof}
\par There is a natural action of $\cc$ on the maps $\CR \to M$ which is given by translation at the source. For a fixed $r$ (less than the $r_1$ above), identifying all the curves given by theorem \ref{receval} which have the same image will not reduce the dimension significantly. 
\par Indeed, let $\jo{I}: \jo{M}_\CR \to \jo{P}(M)$ be defined by taking the image of curve, $\jo{I}(u) = u(\CR)$. $\cc$ acts by reparametrization on $\jo{R}_r(w)$; this leads us to look at the quotient $\jo{R}_r \big( \ell^\infty(\zz;B_{m_*}) \big) / \cc$. Proposition \ref{imcjh} insures us that $\jo{I}$ is locally injective on the quotient.
\par In order to construct a family of curve with different images, it might also be possible to proceed differently. For example, in the case a single gluing, remark \ref{etrang} indicates that letting the parameter vary gives curves of different images. It would then be tempting to use all the parameters $r_i$ not to be equal to the same value $r$, and use this characteristic in order to deduce the maps have different images. However, it is not possible to obtain this result directly from the implicit function theorem. Indeed, if at a point the parameter is $r_i$, the true solution obtained by proposition \ref{til2} is a perturbation in the $\ell^\infty(\wup)$ norm of the order of $\supp{i \in \zz} r_i$. It is also difficult to introduce a measure of the strangling which can be defined on a curve $u$ which is of class $\wup$. For these reasons, evaluation at a point have been used.

\subsubsection{Simple maps.} \label{p:simp}
 Before we look at curves obtained by $\jo{R}$ which possess the same image, we make a digression to show that a careful choice of parameters allows us to show that the map $u$ does not factorize by a quotient of the cylinder. Recall that $u^{r;w}$ is the approximate solution constructed in order to obtain $\jo{R}_r(w)$. The map $u^{r;0}$ is periodic, in the sense that it factorizes as $u^{r;0}:\CR \surj \CR /iN\zz \to M$.
\begin{lemma}\label{upprec}
Let $\pi:\CR \to \CR / i\zz$ be the quotient to the torus. Let $u: \CR \to M$ be a pseudo-holomorphic map sufficiently $C^0$ close to a map $u_0:\CR \to M$ such that $u_0 = u_0' \circ \pi$ and $u_0'$ is injective. Let $\phi: \CR \to \CR'$ and $u':\CR' \to M$ be such that $u = u' \circ \phi$. Then $\phi$ is periodic: $\phi = \phi_2 \circ \phi_1$ where $\phi_1$ is the quotient of the cylinder by a discrete subgroup (without fixed points) of the automorphisms of $\CR$.
\end{lemma}
\begin{proof}
Let $w \in \CR'$ then $\phi^{-1}(w) \subset P_w := u^{-1}(u'(w))$. Since $u_0'$ is injective, $P_w$ is contained in a ball and its translates. Let $B_w \subset \CR /i\zz$ such that $B_w := \inter{z \in P_w}{} \pi(B_\rho (z))$ where $\rho = \nr{u - u_0}_{C^0}$ and $\pi: \CR \to \CR/i\zz$ is the projection on the torus. In particular, $\phi^{-1}(w) \subset \pi^{-1}(B_w)$.
\par We wish to show that $\phi$ is periodic. In order to avoid an accumulation, the number of elements of $\phi^{-1}(w)$ in a component of $B_w$ has to be bounded. Let 
\[
I_{x;k} = i]x-k-\tfrac{1}{2},x+k+\tfrac{1}{2}[ \times \rr/ \zz,
\]
where $x \in \rr$ and $k \in \zz_{>0}$, a piece of the cylinder containing $2k+1$ connected components of $\pi^{-1}(B_w)$. We wish to construct a $J$-holomorphic map which associates to $w \in \CR'$ the mean of its preimages.
\par In order to do so, let us first describe this for a proper and non-constant holomorphic function $f :U \to \CR'$ where $U \subset \CR$. A problem might occur at critical points of $f$. However, locally is $f(z) = a_d z^d + O(z^{d+1})$, there exists a function $g$ such that $f =g^d$ and $g'(0)\neq 0$. In particular, $g$ is invertible. Furthermore, $f(z) = w \ssi g(z) \in \{x| x^d = w\}$. Thus, if $h(z) = \somme{j \geq 0}{} a_j z^j$ is a polynomial and $g^{-1}(x)^j  = \sum_{k \geq 0} b_{j,k}  x^k$ the local expansion of $g$, the sum of the values of $h$ on the preimages of $w$ will be written as
\[
\begin{array} {rll}
\Somme{z \in f^{-1}(w)}{} h(z) 
     &= \Somme{z \in g^{-1}(w^{1/d})}{} h(z) 
     &= \Somme{z \in g^{-1}(w^{1/d})}{} \Somme{j \geq 0}{} a_j z^j\\
     &= \Somme{j \geq 0}{} a_j \Somme{k \geq 0}{} \Somme{x \in w^{1/d}}{} b_{j,k}  x^k  
     &= \Somme{j \geq 0}{} \Somme{k \geq 0}{} a_j d b_{j,dk} w^k  \\
\end{array}
\]
and is a holomorphic function of $w$.
\par In the case of interest to us, when $\phi$ is restricted to $I_{x;k}$, the function which takes the sum of the preimages is well-defined. Let $F_{x;k}(w) = \phi \bv_{I_{x;k}} ^{-1}(w)$. Let $\psi_{x;k}:\CR' \to \CR/i\zz$ be the sequence of function given by
\[
\psi_{x;k}(w) = \frac{1}{2k+1} \somme{z \in F_{x;k}(w)}{} \pi(z) . 
\]
In a neighborhood of $w$ these function are holomorphic. However they present some discontinuities when moving $w$ makes point of the preimage $\phi^{-1}(w)$ leave or enter $I_{x;k}$. In particular, $\psi_{w;k}$ is holomorphic in a neighborhood of $w$ whose size is bounded from below (by the distance from $B_w$ to the boundary of $I_{x;k}$). Furthermore, Since the $\psi_{x;k}$ are holomorphic outside these jumps, it is possible to extract convergent subsequences for each $x$. Note that the size of these discontinuities is bounded by $K_1/k$ for some $K_1 \in \rr_{>0}$. This bound enables to get that, for all $k$,
\[
|\psi_{x;k}-\psi_{x';k}| \leq K_2 \frac{|x-x'|}{k}.
\]
We start by choosing a sequence of points $\{w_i\}$ of $\CR'$ which is dense (but as the $\psi_{w;k}$ do not present jumps near $w$, it would suffice to take a sufficiently small net). Next choose a subsequence $\{n^{(1)}_k\}$ for which $\psi_{w_1;k}$ converges. Then this subsequence $\{n^{(l)}_k \}$ is refined again in another subsequence $\{n^{(l+1)}_k \}$ which makes $\psi_{w_l;n^{(l)}_k}$ converge. The sequence $\{n^{(k)}_k \}$ will make all the $\psi_{w_i;n}$ converge to holomorphic functions. Furthermore, the difference between the jumps tends to $0$, thus the limit of these sequence will not depend on the point $w_i$ chosen. Denote this limit by $\psi$; this is the desired averaging function.
\par The map $\psi \circ \phi: \CR \to \CR/i\zz$ has the property that $|\psi \circ \phi (z) - \pi(z)| < \rho$ since the two points belong to $B_w$. If $\rho$ is less than the injectivity radius of $\CR$, this enables to define a function $f(z)= \psi \circ \phi (z) - \pi(z) \in \cc$ which extends to $X$ and is bounded. Consequently, $\exists c \in \cc$ such that $\psi \circ \phi (z) = z+c$ in $\CR/i\zz$. In particular, $\phi$ is a covering map.
\par Hence the map $\phi$ factors through a quotient of $\cc$ by a discrete subgroup without fixed points of the automorphism group of $\cc$. This subgroup necessarily contains translations, and thus, contains only translations. Since the fundamental group of $\CR$ is abelian, this is the quotient by some group action.
\end{proof}
\par The most restrictive assumption is the fact that $u_0'$ is injective. This means that the pseudo-holomorphic curves $u^k$ must not have any other intersection (or self-intersection) apart from the ones required to make the gluing. 
\par Indeed, by taking $B_{m_*}$ smaller if necessary, lemma \ref{upprec} can then be applied to curves obtained by $\jo{R}$. Suppose that $u= \jo{R}_r(w)$ can be written $u = u' \circ \phi$, where $\phi: \CR \to \CR'$ is a quotient map by a discrete subgroup without fixed points. 
\par These subgroups of the automorphisms of $\CR = \cc / \zz$ possess at most a finite generator (of the form $1 /n$ where $n \in \zz$) and an infinite generator (if it possesses an imaginary part). This is seen by looking at the corresponding discrete subgroup without fixed point of automorphisms of $\cc$, \ie which is a lattice of rank $1$ or $2$.
\par Thus, if $\phi$ is the quotient by an infinite subgroup, then there exists $c \in \cc \setminus \rr$ such that $ u(z+c) = u(z)$. Let $\pi: \cc \surj \cc / (\zz \oplus i N \zz)$, for all $\delta$ there exists an integer $n_c(\delta)$ such that $\pi \big(n_c(\delta) c \big) < \delta$. In other words, if $u$ is periodic, the $w_k$ must be almost periodic: $w_k - w_{k+n_c} < \rho + O(\delta)$ (where $\rho = O(r^{1+\eps})$ is the distance from $u = \jo{R}_r\big((w_k)\big)$ to the approximate solution $u^{r;(w_k)}$). For $r$ is sufficiently small and one of the $w_k$ is apart from the others, $u$ will not be periodic. 
\par Suppose now that $\phi$ is a finite quotient map. Consider a segment $I_k = i[k,k+1] \times \rr/ \zz$ where $u$ is close to the map $u^k$. $u$ is again periodic, but this time in the sense that there exists $c \in \rr$ such that $u(z+c) = u(z)$. Note that $\mu_{k;r}(z) \mapsto \mu_{k;r}(z +c)$ will correspond to an automorphism $\phi'$ of $\cp^1$ which fixes $0$ and $\infty$. Hence $u^k \circ \phi'(z) - u^k(z) \leq \rho$ for $z \in \mu_{k;r}(I_k)$. Consequently if one of the curves $u^k$ is simple and its image is not contained in a small neighborhood, that is of size $O(\rho) = O(r^{1+\eps})$. Thus if a curve is of positive energy, this situation cannot happen. 
\par Recall that a map $u:\SR \to M$ is said simple if when $u = u' \circ \phi$ where $\phi: \SR \to \SR'$ and $u': \SR' \to M$ then $\phi$ is a degree 1 covering map of Riemann surfaces (an automorphism). The previous discussion can be summarized as follows.
\begin{corollary}
If for a $k \in \{1,\ldots, N\}$, $u^k$ is a simple curve of positive energy, there exists a ball $B_{m_*}$ and a $r_0$ such that for all $r<r_0$ and for all $w \in \ell^\infty(\zz;B_{m_*})$ of which a coordinate is at distance at least $\frac{1}{10} \diam B_{m_*}$ from the others, $u = \jo{R}_r(w)$ is a simple map from $\CR$ to $M$.
\end{corollary}

\subsubsection{Mean dimension} \label{s:mdm}

The map $\jo{R}$ given by theorem \ref{receval} will have a infinite dimensional space of maps as its image and mean dimension (as introduced in \cite{Gro} is a natural way to measure the size of this image. Mean dimension is a topological dynamical invariant, and here the group of automorphisms of $\CR$ (which can be identified with $\CR$ itself) acts naturally on $\img(\jo{R})$ by reparametrization at the source. 
\par Let $\jo{M}_\CR$ be the space of pseudo-holomorphic maps $\CR \to M$, and $\jo{M}_{\CR, c}$ be the subspace of those whose differential is bounded (in $C^0$ norm) by $c$. To speak of mean dimension we need to ensure that the metric used makes the space compact. Thus we will use the topology of uniform convergence on compacts. For $u,u': \CR \to M$, we will use the distance
\begin{equation}  \label{cyldist}
d(u,u') = \supp{k \in \zz_{>0}} 2^{-k} \supp{z \in [-k,k] \times S^1} d_M(u(z),u'(z))
\end{equation}
which induces an equivalent topology.
\par To a sequence $\{w_k\}$ of vectors in a small ball $B_{m_*}$ around $\tg_{m_*}M$ we can associate a pseudo-holomorphic cylinder. The distance \eqref{cyldist} between curves associated to $\{w_k\}$ and $\{w'_k\}$ is bounded from below by 
\[ 
d'(w,w') = \supp{k \in \zz_{>0}} 2^{-|k|} \nr{w_k - w'_k}.
\]
If on $B_{m_*}^\zz$ the metric $d'$ above is used, $\jo{R}_r : B_{m_*}^\zz \to \jo{M}_\CR$ does not reduces the distances. Furthermore, $\zz$ acts on $B_{m_*}^\zz$ by shifting and this action is (up to some identification) equivariant if $iN\zz$ is seen as a subgroup of the automorphisms of the cylinder (by translation). More precisely, since deformations only happen on the curve $u^j$ amongst the $N$ curves which form the chain, the shift of an integer in $B_{m_*}^\zz$ will correspond to the translation by $iN$ on the cylinder (here $\CR = \cc / \zz$). 
\par The mean dimension of $(\tg_{m_*} M)^\zz$ for the topology induced by $d'$ (this is the product topology) and the action of $\zz$ is $\dim \tg_{m_*} M = \dim M$ (see \cite{Gro}). Taking into account the covolume of $iN\zz$ yields the following corollary.
\begin{corollary}\label{dimpos}
If there exists an almost complex structure $J$ and a family of curves $u^k$ satisfying the assumptions of theorem \ref{receval} and $\nr{\dd u^k}_{C^0} \leq c$, then the mean dimension of $\jo{M}_{\CR,2c}$ for the action of the automorphism group of the cylinder is at least $\dim M / N >0$.
\end{corollary}
\par There are at least two other ways to obtain a large family of maps. First, suppose there exists a pseudo-holomorphic $u:\cp^1 \to M$ or $u':\CR \to M$. Then one can precompose $u$ or $u'$ by holomorphic maps $\CR \to \cp^1$ or $\CR \to \CR$. This would suffice to generate a family of pseudo-holomorphic which is sufficiently big. However, they would all have their image contained in the image of $u$ or $u'$. Note however that this quantity is finite if we restrict to maps of bounded derivative, and that the image of these maps (as subsets of $M$) is actually rather small.
\par Furthermore, we could be tempted to use directly theorem \ref{threcinf}. To do so, suppose there are $J$-holomorphic maps $u^i: \cp^1 \to M$ where $i=1,\ldots, N$ such that $u^i$ has a point of intersection with $u^j$ if $j \equiv i \pm 1 \modu N$. Being finite, this family gives rise to $\{u^i\}_{i \in \zz}$ which satisfies the assumption of theorem \ref{threcinf} by defining $u^k = u^i$ when $k \equiv i \modu N$ (theorem \ref{threcinf} will again be used on such a finite family of curves). Before gluing those curves $\{u^i\}$ in a cylinder, it is however possible to precompose them by an automorphism fixing the two points which links $u^i$ to its neighbor in the chain, $u^{i-1}$ and $u^{i+1}$. Let $\theta_i$ be these automorphisms fixing $z_{i;0}$ and $z_{i;\infty}$, the maps $\{u^i \circ \theta_i\}$ are another family of maps satisfying the assumption of theorem \ref{threcinf}. By taking all possible $\theta_i$ this will give rise to a family of positive mean dimension (and with bounded differential). It is unfortunately hard to show that the members of this family have distinct images.
\par Thus what theorem \ref{receval} achieves by corollary \ref{dimpos} is to produces family of cylinders (independently of their parametrization) inside the manifold whose mean dimension is positive.

%






\end{document}